\documentclass[12pt,oneside,4apaper]{article}

\usepackage{setspace}
\usepackage[margin=1in]{geometry}
\usepackage[colorlinks=true,citecolor=blue]{hyperref}
\usepackage[mathscr]{euscript}
\usepackage[utf8]{inputenc}
\usepackage[margin=.5in]{caption}

\usepackage{amsmath}%
\usepackage{amsfonts}%
\usepackage{amsthm}
\usepackage{amssymb}%
\usepackage{graphicx}
\usepackage{units}
\usepackage{verbatim}
\usepackage{enumerate}
\usepackage{tikz}
\usepackage{mathtools}
\usepackage{needspace}
\usepackage[inline]{enumitem}

\usepackage{parskip}

\begingroup
    \makeatletter
    \@for\theoremstyle:=definition,remark,plain\do{%
        \expandafter\g@addto@macro\csname th@\theoremstyle\endcsname{%
            \addtolength\thm@preskip\parskip
            }%
        }
\endgroup

\newtheorem{theorem}{Theorem}[section]
\newtheorem{lemma}[theorem]{Lemma}
\newtheorem{corollary}[theorem]{Corollary}

\theoremstyle{remark}

\newtheorem{warning}[theorem]{Warning}

\usepackage{xspace}
\newcommand{\newcmd}[3]{\newcommand{#1}[#2]{#3\xspace}}

\newcommand{\mb}{\mathbb}
\newcommand{\mr}{\mathrm}
\newcmd{\df}{1}{\textbf{\textit{\color{cyan!10!black} #1}}}

\newcommand{\n}{${}$\newline}
\newcommand{\eps}{\varepsilon}
\renewcommand{\phi}{\varphi}

\DeclareMathOperator{\dist}{dist}

\DeclareMathOperator{\orth}{O}
\DeclareMathOperator{\sorth}{SO}

\newcommand{\sphere}{\mathbf{S}}

\DeclareMathOperator{\id}{id}

\newcommand{\projplane}{\mathbf{P}^2}

\DeclareMathOperator{\area}{area}

\DeclareMathOperator{\fradmul}{F_{radmul}}
\DeclareMathOperator{\bump}{bump}

\DeclareMathOperator{\fdivp}{F_{D4p}}
\DeclareMathOperator{\fivpa}{F_{4p\text{-}align}}
\DeclareMathOperator{\fiiipa}{F_{3p\text{-}align}}
\DeclareMathOperator{\fswivel}{F_{swivel}}

\DeclareMathOperator{\fdiip}{F_{D2p}}

\DeclareMathOperator{\fsivp}{F_{S^24p}}
\DeclareMathOperator{\fsiip}{F_{S^22p}}
\DeclareMathOperator{\fstitch}{F_{stitch}}

\DeclareMathOperator{\smagellan}{S_{E}}
\DeclareMathOperator{\samundsen}{S_{A}}
\DeclareMathOperator{\sgagarin}{S_{L}}
\DeclareMathOperator{\sorpheus}{S_{O}}

\DeclareMathOperator{\famundsen}{F_{A}}
\DeclareMathOperator{\fmagellan}{F_{E}}

\newcommand{\im}{\mathfrak{i}}
\DeclareMathOperator{\e}{e}

\begin{document}

\begin{center}
{\Large 
A strong equivariant deformation retraction \\
from the homeomorphism group of the projective plane \\
to the special orthogonal group \par
}

\bigskip

{\large \noindent
Michael~Gene~Dobbins
}


\begin{minipage}{0.85\textwidth}
\raggedright \footnotesize \singlespacing
\noindent
Department of Mathematical Sciences, Binghamton University (SUNY), Binghamton, \newline New York, USA. 
\texttt{mdobbins@binghamton.edu} 
\end{minipage}

\end{center}

\bigskip

\begin{abstract}
This is the third paper in a series on oriented matroids and Grassmannians. 
We construct a $(\orth_3\times\mb{Z}_2)$-equivariant strong deformation retraction from the homeomorphism group of the 2-sphere to $\orth_3$, where the action of $\mb{Z}_2$ is generated by antipodal reflection acting on the right, and $\orth_3$ acts on the left by isometry. 
Quotienting by the antipodal map induces a $\mathrm{SO}_3$-equivariant strong deformation retraction from the homeomorphism group of the projective plane to $\mathrm{SO}_3$. 
The same holds for subgroups of homeomorphisms that preserve the system of null sets. 
This confirms a conjecture of Mary-Elizabeth Hamstrom.
\end{abstract}

\section{Introduction}

Let us denote the multiplicative group on $\{1,-1\}$ by $\mb{Z}_2$ and denote the orthogonal and special orthogonal groups by $\orth_n$ and $\sorth_n$.
Let $\sphere^2$ denote the 2-sphere in $\mb{R}^3$ and $\projplane = \sphere^2/\mb{Z}_2$ denote the real projective plane.  Let $\hom(X)$ denote the group of homeomorphisms from a metric space $X$ to itself with the sup-metric.  
That is, the distance between maps $f,g \in \hom(X)$ is $\dist(f,g) = \sup\{\dist(f(x),g(x)) : x \in X \}$. 
Note that the induced metric topology is the same as the compact-open topology.
Note also that if $f \in \hom(\sphere^2)$ is $\mb{Z}_2$-equivariant, i.e.\ $f(-x) = -f(x)$, then there is an induced map on the projective plane, which we simply denote by $f$, that acts by $f(\{x,-x\}) = \{f(x),f(-x)\}$.  Let $\sorth_3$ act on $\projplane$ in this way.
Our goal is to prove the following.

\begin{theorem}\label{theoremDeformation}
There is a strong $(\orth_3 \times \mb{Z}_2)$-equivariant deformation retraction from $\hom(\sphere^2)$ to $\orth_3$,
where $(Q,s) \in \orth_3 \times \mb{Z}_2$ acts on $f \in \hom(\sphere^2)$ by $(Q,s)f = Q\circ f\circ s$.
\end{theorem}

\begin{corollary}\label{corollaryDeformation}
There is a strong $\sorth_3$-equivariant deformation retraction from $\hom(\projplane)$ to $\sorth_3$,
where $Q \in \sorth_3$ acts on $f \in \hom(\projplane)$ by $Qf = Q\circ f$. 
\end{corollary}

In 1923 Hellmuth Kneser showed that there is a strong deformation retraction from $\hom(\sphere^2)$ to $\orth_3$ \cite{kneser1926deformationssatze}, and later Bjorn Friberg gave another more elementary proof \cite{friberg1973topological}.  
Dobbins showed that a strong deformation retraction from $\hom(\sphere^2)$ to $\orth_3$ can be made to be $\orth_3$-equivariant \cite{dobbins2021grassmannians}.
Mary-Elizabeth Hamstrom 
showed that $\hom(\projplane)$ and $\sorth_3$ have isomorphic homotopy groups. 
Hamstrom further conjectured that $\hom(\projplane)$ deformation retracts to the Lie group of $\sorth_3$ acting on $\projplane$ \cite[p.~43]{hamstrom1965homotopy}, which is confirmed by Corollary \ref{corollaryDeformation}.

Let $\hom_\mathrm{N}(X)$ denote the subgroup of $f \in \hom(X)$ such that for all $A \subset X$, $A$ is a null set if and only if $f(A)$ is a null set.  
In other words, both $f$ and $f^{-1}$ have the Luzin N property. 
In this case we say $f$ preserves nullity. 
We also prove analogous results for these groups.

\begin{theorem}\label{theoremDeformationN}
There is a strong $(\orth_3 \times \mb{Z}_2)$-equivariant deformation retraction from $\hom_\mathrm{N}(\sphere^2)$ to $\orth_3$.
\end{theorem}

\begin{corollary}\label{corollaryDeformationN}
There is a strong $\sorth_3$-equivariant deformation retraction from $\hom_\mathrm{N}(\projplane)$ to $\sorth_3$.
\end{corollary}

For us, the subgroup $\hom_\mathrm{N}(\sphere^2)$ has the advantage that 
for a fixed disk $D \subset \sphere^2$, the area of $f(D)$ is continuous as a function on $f \in \hom_\mathrm{N}(\sphere^2)$, which is not the case on $\hom(\sphere^2)$.  
See Lemma \ref{lemma-area-convergence} below.

\subsection{Motivation from oriented matroids}

Groups of automorphisms are of fundamental interest. 
In this case, however, the author was motivated by a conjecture from combinatorics, which in turn has applications for working with vector bundles.  
Some background on the conjecture that the author wrote in the previous papers in the series is repeated in this section \cite{dobbins2021continuous,dobbins2021grassmannians}. 
The material in this subsection is not needed to understand the rest of the paper. 

An oriented matroid is a combinatorial analog to a real vector space, but where we only keep track of sign information.  One way to obtain an oriented matroid is as the set of all sequences of signs of all the vectors in a vector subspace of $\mb{R}^n$.  For example, the oriented matroid obtained from the space $\{(x,y) \in \mb{R}^2 : x+y = 0\}$ is given by the set  
$\{(+,-),(0,0),(-,+)\}$. 
However, oriented matroids are defined by purely combinatorial axioms, and not all oriented matroids are obtained in this way.  One way the Grassmannian is defined is as the set of all $k$-dimensional vector subspaces of a given vector space, and this set is given a metric.  Oriented matroids come equipped with analogs of dimension and subspace, and an OM-Grassmannian is a finite simplicial complex that is defined analogously from an oriented matroid. 
Nicolai Mnëv and Günter Ziegler conjectured that each OM-Grassmannian of a realizable oriented matroid is homotopy equivalent to the corresponding real Grassmannian \cite[Conjecture 2.2]{mnev1993combinatorial}.

Gaku Liu showed that this conjecture does not hold in general; i.e.\ there is an OM-Grassman\-nian of a realizable oriented matroid that is not homotopy equivalent to the corresponding real Grassmannian \cite{liu2020counterexample}.  
Although the full conjecture is false, special cases remain open, such as for MacPhersonians, which are the OM-Grassmannians of the oriented matroids corresponding to $\mb{R}^n$, and in particular for the rank 3 case, which is the analog of the Grassmannian consisting of 3-dimensional subspaces of $\mb{R}^n$.
Mnëv and Ziegler had listed the rank 3 case as a theorem with the expectation that a proof would later appear in Eric Babson's Ph.D.\ thesis, but then it did not \cite{mnev1993combinatorial,babson1993combinatorial}.  Also, an erroneous proof that each MacPhersonian is homotopy equivalent to the corresponding Grassmannian in all ranks was published and then retracted \cite{biss2003homotopy,biss2009erratum}.


Jim Lawrence showed that oriented matroids can be characterized as the combinatorial cell decompositions of the sphere arising from essential pseudosphere arrangements \cite{folkman1978oriented}.  Lawrence's theorem provides a topological model for the combinatorial axioms of oriented matroids.  In the rank 3 case, these are pseudocircle arrangements, which are collections of oriented simple closed curves in the 2-sphere such that every pair of curves either coincide or intersect at exactly 2 points, in which case any third curve either separates the 2 points or passes though both points.  An arrangement is said to be essential when no single point is contained in all pseudospheres. 

The first paper in the series 
introduced spaces of weighted essential pseudosphere arrangements called pseudolinear Grassmannians, and showed that each rank 3 pseudolinear Grassmannian is homotopy equivalent to the corresponding real Grassmannian.  
The pseudolinear Grassmannians serve as an intermediate space between the Grassmannians and MacPhersonians, and the main results of the first paper represent a step toward showing that the conjecture holds for the rank 3 MacPhersonians by showing homotopy equivalence between this intermediate space and one side \cite{dobbins2021grassmannians}.  
The second paper provided a crucial tool needed for the present paper, which will be used in Subsection \ref{subsectionUntangling} \cite{dobbins2021continuous}.

The present paper takes another step toward showing the rank 3 MacPhersonian case of the conjecture, by providing a tool for replacing pseudolinear Grassmannians with nicer spaces, namely spaces of \emph{antipodally symmetric} weighted pseudocircle arrangements where every pseudocircle has \emph{area 0}.  The first advantage is that we can disregard the last property of pseudocircle arrangements, namely that a third curve must separate or pass though the points where two other curves intersect, since this property is already guaranteed to hold in the symmetric case.  The second advantage is that the area of the cells of an arrangement vary continuously with respect to the arrangement. 
The present paper will not go into the precise details on how of Theorem \ref{theoremDeformationN} is applied to pseudocircle arrangements.





\subsection{Organization}

A central idea in the proof of Theorem \ref{theoremDeformation} is to find a simple closed curve that depends continuously on the initial homeomorphism $f \in \hom(\sphere^2)$ and then continuously deform that curve to a great circle.  
Then, each hemisphere on either side of the resulting great circle can be dealt with separately. 
To do so, we have to extend the deformation of the closed curve to a deformation of the homeomorphism $f$, which we do by using various canonical versions of the Jordan-Schoenflies theorem.
For this, we use parameterizations of the sphere or of disks on the sphere that are each defined explicitly using complex analysis and the Riemann mapping theorem.
In Section \ref{sectionParameterizations}, we present these parameterizations. 
In Section \ref{sectionDeformation}, we prove Theorem \ref{theoremDeformation}.
In Section \ref{sectionDeformationN}, we modify the proof from Section \ref{sectionDeformation} to arrive at Theorem \ref{theoremDeformationN}.

\subsection{Definitions and Notation}

Let $\mathbf{D}$ be the closed unit disk and $\sphere^1$ be the unit circle in the complex plane.
Let $e_1,e_2,e_3$ be the standard basis vectors in 3-space, and let $e_{-i} = -e_i$. 
We will use $\im$ for the imaginary unit, and we let $\overline{z} = a -b\im$ be the complex conjugate of $z = a +b\im$. 
Let $\overline{\mb{R}}$ and $\overline{\mb{C}}$ be the one point compactifications of the real line and the complex plane by adjoining $\infty$.
We denote the derivative of a function $f$ by $\partial f$ or by $\partial_xf(x)$ and we denote the boundary of a topological disk $C$ by $\partial C$.
We will generally use square brackets around functions and round brackets around arguments passed to functions.
We use this for partial function application. 
For example, given a function $f : W \times X \to Y$ and $w \in W$, 
we have $f(w) : X \to Y$ by $[f(w)](x) = f(w,x)$. 
We will also use square brackets when composing, inverting, and adjoining functions.
For example, 
given $f : W \times X \to Y$, $g : Z \to Y$, and $h : Z \to W$ we write $x = [[f(w)]^{-1}\circ g](z)$ for the value $x$ such that $f(w,x) = g(z)$, provided that $f(w) : X \to Y$ is invertible. 
Also, $[g,h](z) = (g(z),h(z))$. 
We may distinguish multiplicative inverse by using round brackets, i.e., $(x)^{-1} = \frac1x$. 
We let $[x \mapsto \Phi]$ denote the anonymous function that substitutes its input for the value of a variable $x$ in the formula $\Phi$.
For example $[z \mapsto \overline{z}]$ denotes the function that returns the complex conjugate of its input. 
When composing $f$ with multiplication by a constant $c$, we simply write $cf = [x \mapsto cf(x)]$ and $fc = [x \mapsto f(cx)]$, and similarly we write $Qf$ and $fQ$ for $Q \in \orth_3$ acting on the sphere. 
We denote real intervals by $(a,b]_\mb{R} = \{x \in \mb{R} : a < x \leq b\}$ for bounds $a,b \in \mb{R}$ with any combination of round or square brackets for (half) open or closed intervals.

The \df{Fréchet distance} between a pair of Jordan curves is defined by 
\[ \dist_\mr{F} (\gamma_1,\gamma_0) = \inf_{\phi_1,\phi_0} \sup_x \|\phi_1(x) -\phi_0(x) \|  \]
where the $\phi_i : \sphere^1 \to \gamma_i$ are homeomorphisms. 
We will also use Hausdorff distance, which is a coarser metric than Fréchet distance.
\df{Hausdorff distance} is defined by 
\[
\dist_H(\gamma_1,\gamma_0) = \inf \left\{ \delta : \gamma_1 \subseteq \gamma_0 \oplus \delta, \gamma_0 \subseteq \gamma_1 \oplus \delta \right\}. 
\]
The \df{boundary Fréchet distance} between a pair of topological disks is the Fréchet distance between the boundaries of the disks.
We say a curve is \df{rectifiable} when it has finite length and has \df{null area} when its area is 0.

We say a map is \df{internally conformal} when the restriction of the map to the interior of its domain is conformal. 
We will use the following two extensions of the Riemann mapping theorem. 
Carathéodory's mapping theorem says that, for every Jordan curve in $\mb{C}$, there is an internally conformal homeomorphism from the closed disk to the region bounded by the curve.
Radó's theorem says that, for Jordan curves $\gamma_k$ 
and maps $h_k$ as in the Carathéodory's mapping theorem, 
if 
$h_k(0)$ is fixed and $\partial h_k(0) > 0$ for all $k$,
and $\gamma_k \to \gamma_\infty$ in Fréchet distance,
then $h_k \to h_\infty$ uniformly \cite{pommerenke1992boundary}.

\section{Parameterizations}\label{sectionParameterizations}

In this section we first introduce a notion of convergence called invariable convergence.
We will then present several parameterizations of the sphere and disks on the sphere and their properties.
These parameterizations will be used in the next section to construct a deformation retraction.

\subsection{Invariable convergence}

A vitally important property for us is continuity. 
Here we will use parameterizations of a class of topological disks, and will need these parameterizations to depend continuously on the boundary of the disk and on some specified points on the boundary. 
What's more, we will need the inverse map of the parameterization to depend continuously on the boundary and the specified points. 
In this case, uniform convergence is not applicable, since the domain of the map can change.
To deal with this, we introduce a related notion of convergence called invariable convergence that is relevant to sequences of maps where the domain changes.


Let $X,Y$ be metric spaces with metrics $\dist_X,\dist_Y$, and let $X_k \subset X$ for $k \in \{1,\dots,\infty\}$.
We say a sequence of functions $f_k : X_k \to Y$ converges \df{invariably} to $f_\infty : X_\infty \to Y$ when the following holds:  $X_k \to X_\infty$ in Hausdorff distance, and for all $\eps > 0$, there is $K(\eps)$ and $\delta(\eps) > 0$ such that for all $k \in \mb{N}$, all $x_k \in X_k$, and all $x_\infty \in X_\infty$, if $k \geq K(\eps)$ and $\dist_X(x_k,x_\infty) < \delta(\eps)$ then $\dist_Y(f_k(x_k), f_\infty(x_\infty)) < \eps$.
Note that $K$ and $\delta$ cannot depend on $x_k$ or $x_\infty$.

\begin{lemma}\label{lemma-invariable-convergence}
Let $X,Y,Z$ be compact metric spaces, and let $f_k : X_k \to Y_k$ and $g_k: Y_k \to Z$ be maps defined on $X_k \subset X$ and $Y_k \subset Y$ for $k \in \{1,\dots,\infty\}$.
\begin{enumerate}
\item \label{item-fixed-domain}
For $X_k = X$ fixed, \n
$f_k \to f_\infty$ invariably if and only if 
$f_k \to f_\infty$ uniformly and $f_\infty$ is continuous.
\item \label{item-inverse}
If $f_k : X_k \to Y$ is a homeomorphic embedding of $X_k$ in $Y$ and $f_k \to f_\infty$ invariably, then $f_k^{-1} \to f_\infty^{-1}$ invariably. 
\item \label{item-composition}
If $f_k \to f_\infty$ and $g_k \to g_\infty$ invariably, then $g_k \circ f_k \to g_\infty \circ f_\infty$ invariably. 
\end{enumerate}
\end{lemma}




\begin{proof}[Proof of Lemma \ref{lemma-invariable-convergence} part \ref{item-fixed-domain}]
Suppose $f_k \to f_\infty$ uniformly and $f_\infty$ is continuous.
Then, by the Heine-Cantor theorem, $f_\infty$ is uniformly continuous.
Consider $\eps > 0$.
Since $f_k \to f_\infty$ uniformly, there is $K$ such that for all $x \in X$ and all $k \geq K$,
$\dist_Y(f_k(x),f_\infty(x)) < \eps/2$.
Since $f_\infty$ is uniformly continuous, there is $\delta$ such that for all $x,x' \in X$, if $\dist_X(x,x') < \delta$, then $\dist_Y(f_\infty(x),f_\infty(x')) < \eps/2$.
Hence, for all $x_k,x_\infty \in X$, if $k \geq K$ and $\dist_X(x_k,x_\infty) < \delta$, then  
\[\dist_Y(f_k(x_k), f_\infty(x_\infty)) 
\leq \dist_Y(f_k(x_k), f_\infty(x_k)) + \dist_Y(f_\infty(x_k), f_\infty(x_\infty))  
< \eps,\]
so $f_k \to f_\infty$ invariably.

To show the implication in the other direction, suppose $f_k \to f_\infty$ invariably.
First, fix an arbitrary point $x_\infty \in X$ and let $x_k = x_\infty$ for all $k$. Then, the definition of invariable convergence in this case is equivariant to uniform convergence.
Next, fix $\eps > 0$ and $\widetilde x \in X$ such that $\dist_X(x_\infty,\widetilde x) < \delta(\eps)$ where $\delta$ is given in the definition of invariable convergence.  Then, in the limit as $k \to \infty$, we have $f_k(x_k) = f_k(x_\infty) \to f_\infty(x_\infty)$, so $\dist (f_\infty(x_\infty),f_\infty(\widetilde x)) < \eps$, which means that $f_\infty$ is continuous.
\end{proof}

\begin{proof}[Proof of Lemma \ref{lemma-invariable-convergence} part \ref{item-inverse}]
Suppose for the sake of contradiction that $f_k^{-1}$ does not converge invariably to $f_\infty^{-1}$ under the hypotheses of the lemma.
Then, there would be $\eps > 0$ such that for all $K$ and all $\delta>0$, there is some $k \geq K$ and $y_k \in f(X_k)$ and $y_\infty \in f(X_\infty)$ such that $\dist_Y(y_k,y_\infty) < \delta$ and $\dist_X(f_k^{-1}(y_k),f_\infty^{-1}(y_\infty)) \geq \eps$.
This means that we may assume there is some sequence $x_k$ such that $y_k = f_k(x_k)$ converges to $y_\infty = f_\infty(x_\infty)$ and $\dist_X(x_k,x_\infty) > \eps$; otherwise we could restrict to a subsequence where $\delta \to 0$ and then we could restrict to a subsequence where $y_k$ converges, since $Y$ is compact. 
We may also assume that $x_k$ converges to a point $\widetilde x$, since $X$ is compact. 
Since $f_k \to f_\infty$ invariably and $x_k \to \widetilde x$, we have 
$y_{k} = f_{k}(x_{k}) \to f_{\infty}(\widetilde x)$.
Since $y_k \to y_\infty$, we have $f_{\infty}(\widetilde x) = y_\infty$, so $\widetilde x = f_\infty{}^{-1}(y_\infty) = x_\infty$, but that contradicts our choice of $x_{k}$ as a sequence bounded away from $x_\infty$.
\end{proof}

\begin{proof}[Proof of Lemma \ref{lemma-invariable-convergence} part \ref{item-composition}]
Let $\delta_\mathrm{f}$, $\delta_\mathrm{g}$ and $K_\mathrm{f}$, $K_\mathrm{g}$ be as in the definition of invariable convergence for $f_k$ and $g_k$ respectively. 
Consider $\eps > 0$ and $k \geq \max(K_\mathrm{f}(\eps),K_\mathrm{g}(\eps))$ and $x \in X_k$ and $x' \in X_{\infty}$ such that $\dist_X(x, x') < \delta_\mathrm{f}(\delta_\mathrm{g}(\eps))$.  Then, 
$ \dist_Y( f_k(x) , f_{\infty}(x') ) < \delta_\mathrm{g}(\eps) $, so 
$ \dist_Z( g_k \circ f_k(x) , g_{\infty} \circ f_{\infty}(x') ) < \eps $.
\end{proof}

\begin{lemma}\label{lemma-uniform-inverse}
Let $f_k \to f_\infty : X \to Y$ be a sequence of homeomorphic embeddings of a compact metric space $X$ in a compact metric space $Y$ converging uniformly to a homeomorphic embedding, and let $y_k \in f_k(X)$. If $y_n \to y_\infty \in Y$, 
Then $f_k{}^{-1}(y_k) \to f_\infty{}^{-1}(y_\infty)$.
\end{lemma}

\begin{proof}
By Lemma \ref{lemma-invariable-convergence} part \ref{item-fixed-domain}, $f_k \to f_\infty$ invariably, so by part \ref{item-inverse}, $f_k^{-1} \to f_\infty^{-1}$ invariably, so $f_k{}^{-1}(y_k) \to f_\infty{}^{-1}(y_\infty)$.
\end{proof}

\begin{lemma}\label{lemmasupconvergence}
Let $\Omega, X, Y, Z$ be compact metric spaces spaces, and let $f : \Omega \to \hom(X,Y)$ and $g : \Omega \to \hom(Y,Z)$ be continuous in the sup-metric.  
Then, 
$g(\omega)\circ f(\omega)$ and $[f(\omega)]^{-1}$ depend continuously on $\omega$ in the sup-metric.
\end{lemma}

\begin{proof}
Consider $\omega_k \in \Omega$ such that $\omega_k \to \omega_\infty$. 
By Lemma \ref{lemma-invariable-convergence}, $f(\omega_k) \to f(\omega_\infty)$ and $g(\omega_k) \to g(\omega_\infty)$ converge invariably, so $g(\omega_k)\circ f(\omega_k) \to g(\omega_\infty)\circ f(\omega_\infty)$ and 
$[f(\omega_k)]^{-1} \to [f(\omega_k)]^{-1}$ invariably, so $g(\omega)\circ f(\omega)$ and $[f(\omega)]^{-1}$ are continuous as functions of $\omega$.
\end{proof}

\subsection{A disk with four points on the boundary - $\fdivp$}

Here we construct a parameterization $\fdivp$ of a topological disk that depends continuously on the boundary of the disk and on 4 distinct points on the boundary.

\vbox{
\begin{lemma}\label{lemma-ivpointmap}
Let $S$ be a simple closed curve in $\sphere^2$, and let $C$ be one of the closed regions bounded by $S$, and let $p_{1},p_{2},p_{-1},p_{-2} \in S$ be distinct points counter-clockwise around $C$ as viewed from outside the sphere. 
Then, there is a unique internally conformal homeomorphism $f = \fdivp(S,p_1,p_2,p_{-1},p_{-2}) : \mathbf{D} \to C$ such that 
$f(1) = p_1$, $f(-1) = p_{-1}$, and $f^{-1}(p_{-2}) = -f^{-1}(p_{2})$.
Also, the following holds. 
\begin{enumerate}
\item \label{item-ivpointmapcontinuity}
In the sup-metric, $f$  depends continuously on $S$ in Fréchet distance and $p_{1},p_{2},p_{-1},p_{-2}$. 
\item \label{item-ivpointmapsorth}
For $Q \in \sorth_3$, 
$ \fdivp(Q(S,p_1,p_2,p_{-1},p_{-2})) = Qf  $.
\item \label{item-ivpointmapnegate}
$ \fdivp(-(S,p_1,p_2,p_{-1},p_{-2});z) = -f(\overline{z})  $.
\item \label{item-ivpointmapswap}
$ \fdivp(S,p_{-1},p_{-2},p_1,p_2;z) = f(-{z}) $.
\item \label{item-ivpointmapstereo}
$ \fdivp(e_3^\bot,e_1,e_2,e_{-1},e_{-2})$ 
is stereographic projection through $e_{-3}$. 
\end{enumerate}
\end{lemma}
}

Before proving Lemma \ref{lemma-ivpointmap}, we will define a pair of conformal automorphisms of the unit disk.
Note that the conformal automorphisms of the unit disk are the linear fractional transformations that send the disk to itself.
One automorphism, $\fiiipa$, depends on 3 points on the boundary, and the other, $\fivpa$, depends on 4 points on the boundary.  We will show that these automorphisms depend continuously on the 3 or 4 points. 
The map $\fdivp$ implied by Lemma \ref{lemma-ivpointmap} will use the automorphism with 4 points, which will use the automorphism with 3 points. 

\begin{lemma}\label{lemma-iiipointalign}
Let $u_1,u_\im,u_{-1} \in \mathbf{S}^1$ be distinct points in counter-clockwise order.
Then, there is a unique conformal automorphism $f = \fiiipa(u_1,u_\im,u_{-1})$ of the disk 
such that $f(u_n) = n$ for $n \in \{1,\im,-1\}$. 
Moreover, $f$ in the sup-metric depends continuously on $u_1,u_\im,u_{-1}$. 
\end{lemma}

\begin{proof}
Let 
\[ f(z) = \frac{(z-u_{-1})(u_\im-u_1)(\im+1)+(z-u_1)(u_\im-u_{-1})(\im-1)}{(z-u_{-1})(u_\im-u_1)(\im+1)-(z-u_1)(u_\im-u_{-1})(\im-1)}. \]
To verify that $f$ has the defining properties, observe that $f(u_1) = 1$, $f(u_\im) = \im$, and $f(u_{-1}) = -1$.  
Since linear fractional transformations send circles to circles, and 3 points determine a circle, and 
$f$ respectively sends $u_1,u_\im,u_{-1}$ to $1,\im,-1$, we have that $f$ sends the unit circle $\sphere^1$ to itself.
Also, since the order of the points around the circle is the same, namely counter-clockwise, and linear fractional transformations are orientation preserving, $f$ sends $\mathbf{D}$ to itself.
Thus, $f$, as given by the formula above, satisfies the defining properties of $\fiiipa(u_1,u_\im,u_{-1})$ in the lemma.

To verify uniqueness, consider another map $f_0$ satisfying the hypotheses of the lemma,
and let $f_1 = f_0 \circ f^{-1}$.
Then, $f_1$ is a linear fractional transformation with fixed values $1$, $\im$, and $-1$, so $f_1$ is the identity map, which means $f_0 = f$.

To show continuity, observe that $\fiiipa$ is smooth as a function of 4 variables, so the restriction of $\fiiipa$ to a compact subset of the domain is always uniformly continuous.  
We can always find a compact neighborhood of a point $(u_1,u_\im,u_{-1};z)$ that is a product with $\mathbf{D}$, so $\fiiipa$ is continuous in the sup-metric over $\mathbf{D}$ as a function of 3 variables.
\end{proof}

\begin{lemma}\label{lemma-ivpointalign}
Let $u_1,u_2,u_{-1},u_{-2} \in \mathbf{S}^1$ be distinct points in counter-clockwise order.
Then, there is a conformal automorphism $f = \fivpa(u_1,u_2,u_{-1},u_{-2})$ of the disk 
such that $f(u_1) = 1$,  $f(u_{-1}) = -1$, and  $f(u_{-2}) = -f(u_{2})$. 
Moreover, $f$ in the sup-metric depends continuously on $u_1,u_2,u_{-1},u_{-2}$. 
\end{lemma}

\begin{proof}
Let 
\begin{align*}
r &= \frac{(u_1-u_2)(u_{-1}-u_{-2})}{(u_1-u_{-2})(u_{-1}-u_2)}. \\
\intertext{Note that $r$ is the cross-ratio of $(u_{1},u_{-1};u_{2},u_{-2})$.
Since these 4 values appear on a common circle, $r$ is real, and because the pair $u_{1},u_{-1}$ and $u_{2},u_{-2}$ appear in alternating order around the circle, we have $r< 0$. Let}
w &= \frac{1+\sqrt{r}}{1-\sqrt{r}} \quad \text{where $\sqrt{r}$ is in the upper half-plane}. \\
\intertext{Observe that $\sqrt{r}$ is on the upper ray of the imaginary axis by our choice of branch of the square root, so $w$ is on the upper half of the unit circle.
Let $f$ be the linear fractional transformation that sends $u_1,u_2,u_{-1}$ respectively to $1,w,{-1}$, namely }
f(z) &= [\fiiipa(1,w,-1)^{-1}\circ \fiiipa(u_1,u_2,u_{-1})](z) \\
&= \frac{(z-u_{-1})(u_2-u_1)(w+1)+(z-u_1)(u_2-u_{-1})(w-1)}{(z-u_{-1})(u_2-u_1)(w+1)-(z-u_1)(u_2-u_{-1})(w-1)}. 
\end{align*}

To verify that $f$ has the defining properties, we can immediately observe that $f(u_1) = 1$, $f(u_2) = w$, and $f(u_{-1}) = -1$.  
Since $u_1,u_2,u_{-1}$ are sent to points on the unit circle, and $f$ is a linear fractional transformation, we have that $f(\sphere^1) = \sphere^1$.
Furthermore, since $u_1,u_2,u_{-1}$ are sent to points in the same order around the unit circle, namely counter-clockwise, we have that $f(\mathbf{D}) = \mathbf{D}$. 
It only remains for us to verify that $f(u_{-2}) = -w$ by computation. 
We have 
$\frac{w-1}{w+1} = r^{\nicefrac{1}{2}}$ 
and
\begin{align*} 
(u_1-u_2)(u_{-1}-u_{-2}) &= r^{\nicefrac{1}{2}} \sqrt{(u_1-u_2)(u_{-1}-u_{-2})(u_1-u_{-2})(u_{-1}-u_2)}, \\
(u_1-u_{-2})(u_{-1}-u_{2}) &= r^{-\nicefrac{1}{2}} \sqrt{(u_1-u_2)(u_{-1}-u_{-2})(u_1-u_{-2})(u_{-1}-u_2)}, 
\end{align*}
so 
\begin{align*} 
f(u_{-2}) 
& = \frac{(u_1-u_2)(u_{-1}-u_{-2})(w+1)+(u_1-u_{-2})(u_{-1}-u_2)(w-1)}{(u_1-u_2)(u_{-1}-u_{-2})(w+1)-(u_1-u_{-2})(u_{-1}-u_2)(w-1)}  \\
& = \frac{r^{\nicefrac{1}{2}}(w+1)+r^{\nicefrac{-1}{2}}(w-1)}{r^{\nicefrac{1}{2}}(w+1)-r^{\nicefrac{-1}{2}}(w-1)} \cdot \frac{\sqrt{(u_1-u_2)(u_{-1}-u_{-2})(u_1-u_{-2})(u_{-1}-u_2)}}{\sqrt{(u_1-u_2)(u_{-1}-u_{-2})(u_1-u_{-2})(u_{-1}-u_2)}}  \\
& = \frac{(w-1)+(w+1)}{(w-1)-(w+1)}  \\
& = -w.
\end{align*}
Thus, $f$, as given by the formula above, satisfies the defining properties of $\fivpa(u_1,u_2,u_{-1},u_{-2})$ in the lemma.

To verify uniqueness, consider another map $f_0$ satisfying the hypotheses of the lemma,
and let $f_1 = f_0 \circ f^{-1}$.
Then $f_1$ is a linear fractional transformation that fixes $1$ and $-1$ and sends $\sphere^1$ to itself, so $f_1$ must also send $\overline{\mb{R}}$ to itself.
Also, $f_1$ sends the line through $f(u_2),0,f(u_{-2})$ to the line though $f_0(u_2),0,f_0(u_{-2})$, 
and these lines are distinct from $\overline{\mb{R}}$, so $f_1$ must send $\{0,\infty\}$ to $\{0,\infty\}$,
and since $f_1$ sends $\mathbf{D}$ to itself, $0$ must be another fixed point of $f_1$, so $f_1$ is the identity map. 
Thus, $f_0 = f$. 
 
To show continuity, observe that $\fivpa$ is smooth as a function of 5 variables, so $\fivpa$ is locally uniformly continuous, so $\fivpa$ is continuous in the sup-metric over $\mathbf{D}$ as a function of 4 variables.
\end{proof}

\begin{proof}[Proof of Lemma \ref{lemma-ivpointmap}]
By the Riemann and Carathéodory mapping theorems, there exists an internally conformal homeomorphism $g$ from $\mathbf{D}$ to $C$.  Fix $g$, and let 
\[ f = g \circ \left(\fivpa(u_1,u_2,u_{-1},u_{-2})\right)^{-1}\]
where $u_n = g^{-1}(p_n)$.
Observe that $f$ satisfies the defining properties of $\fdivp(S,p_1,p_2,p_{-1},p_{-2})$ in the lemma.

To verify uniqueness, consider another map $\widetilde f$ satisfying the hypotheses of the lemma, and let 
$h = \widetilde f \circ f^{-1}$.
Then, $h : \mathbf{D} \to \mathbf{D}$ is an internally conformal homeomorphism such that  
$h(1) = 1$ and $h(-1) = -1$.
By the Schwarz-Pick theorem, $h$ is a linear fractional transformation. 
Also, $h$ sends the line through $f(p_2),0,f(p_{-2})$ to the line though $\widetilde f(p_2),0,\widetilde f(p_{-2})$, 
and these lines are distinct from $\overline{\mb{R}}$, so $0$ must be another fixed point of $h$, so $h$ is the identity map. 
Thus, $\widetilde f = f$.

To show part \ref{item-ivpointmapcontinuity}, continuity, 
consider sequences $S_k,p_{1,k},p_{2,k},p_{-1,k},p_{-2,k}$ that respectively converge to $S_\infty,p_{1,\infty},p_{2,\infty},p_{-1,\infty},p_{-2,\infty}$ in the space of objects where the hypotheses of the lemma are satisfied, and define $C_k$ analogously. 
Let $f_k = \fdivp(S_k,p_{1,k},p_{2,k},p_{-1,k},p_{-2,k})$. 
Let us identify $\sphere^2$ with $\overline{\mb{C}}$ preserving angle and orientation in such a way that $0 \in C_\infty^\circ$, $\infty \not\in C_\infty$, $\partial f_\infty(0) > 0$.  Then, $f_k$ is holomorphic on the interior of $\mathbf{D}$.

Let $g_k$ be the internally conformal homeomorphism from $\mathbf{D}$ to $C_k$ such that $g_k(0) = f_\infty(0)$ and $\partial g_k(0) > 0$.
Since $S_k \to S_\infty$ in Fréchet distance, by Radó's theorem \cite[Theorem 2.11]{pommerenke1992boundary}, $g_k \to g_\infty$ uniformly.
Let $u_{n,k} = g_k^{-1}(p_{n,k})$.
Then, $u_{n,k} \to u_{n,\infty}$ by Lemma \ref{lemma-uniform-inverse},
so by Lemmas \ref{lemma-ivpointalign} and \ref{lemma-invariable-convergence}, 
\[ f_k = g_k \circ \left(\fivpa(u_{1,k},u_{2,k},u_{-1,k},u_{-2,k})\right)^{-1} \to f_\infty \quad \text{uniformly.} \]
Thus, $\fdivp$ is continuous in the sup-metric over $\mathbf{D}$ as a function of the 5 variables $S,p_1,p_2,p_{-1},p_{-2}$, so part \ref{item-ivpointmapcontinuity} holds.

The map $ Q f$ is an internally conformal homeomorphism from $\mathbf{D}$ to $Q C$ that sends $1,-1$ respectively to $Q p_1, Q p_{-1}$, and  $[Qf]^{-1}(Q p_{-2}) = - [Qf]^{-1}(Q p_{2}) $, 
and since 
$\fdivp(Q(S,\dots,p_{-2}))$ uniquely satisfies these conditions, part \ref{item-ivpointmapsorth} holds.

The map $h = [z \mapsto -\fdivp(S,\dots,p_{-2};\overline{z})] = [z \mapsto -z] \circ \fdivp(S,\dots,p_{-2}) \circ [z \mapsto \overline{z}]$ is a homeomorphism from $\mathbf{D}$ to $-C$, and both $[z \mapsto \overline{z}]$ and $[z \mapsto -z]$ reverse orientation, and all three components of the map preserve angles on the interior, so the map $h$ is internally conformal. 
Also, the map $h$ sends $1,-1$ respectively to $-p_{1}$ and $-p_{-1}$, and 
$h^{-1}(-p_{-2}) = \overline{f^{-1}(p_{-2})} = -\overline{f^{-1}(p_{2})} = -h^{-1}(-p_{2})$. 
Since $\fdivp(-(S,\dots,p_{-2}))$ uniquely satisfies these conditions,
part \ref{item-ivpointmapnegate} holds.

The map $h = [z \mapsto f(-z)]$ is an internally conformal homeomorphism from $\mathbf{D}$ to $C$ that sends $1,-1$ respectively to $p_{-1}$ and $p_1$, and  $h^{-1}(p_{2}) = - h^{-1}(p_{-2})$.
Since $\fdivp(S,p_{-1},p_{-2},p_1,p_2)$ uniquely satisfies these conditions,
part \ref{item-ivpointmapswap} holds.

Finally, stereographic projection though the point $e_{-3}$ from $\mathbf{D}$ to $\sphere^2$ is conformal and  sends $\sphere^1,1,\im,-1,-\im$ respectively to $e_3^\bot,e_1,e_2,e_{-1},e_{-2}$,
so part \ref{item-ivpointmapstereo} holds.  
\end{proof}

\subsection{A disk with two points on its null area boundary - $\fdiip$}

We call a curve a \df{null area} curve when the curve itself has area 0.  
Here we construct a parameterization $\fdiip$ of a topological disk that depends continuously on the boundary of the disk and on 2 distinct points on the boundary, provided that that boundary is a null area curve.

\begin{lemma}\label{lemma-iipointmap}
Let $C$ be a closed region in $\sphere^2$ bounded by a null area simple closed curve $S$, and let $p_{1},p_{-1} \in S$ be distinct points on the boundary of $C$.
Then, there is a unique internally conformal homeomorphism $f = \fdiip(C,p_1,p_{-1}) : \mathbf{D} \to C$ such that $f(1) = p_1$, $f(-1) = p_{-1}$, and the regions $f(\mathfrak{i}\mathbf{H}\cap\mathbf{D})$ and $f(-\mathfrak{i}\mathbf{H}\cap\mathbf{D})$ have equal area.
Also, the following hold.
\begin{enumerate}
\item \label{item-iipointmapcontinuity}
In the sup-metric, $f$ depends continuously on $C$ in boundary Fréchet distance and $p_1,p_{-1}$. 
\item \label{item-iipointmapsorth}
For $Q \in \sorth_3$,
$\fdiip(Q(C,p_1,p_{-1})) = Q f.$ 
\item \label{item-iipointmapnegate}
$\fdiip(-(C,p_1,p_{-1});z) = -f(\overline{z})$.
\item \label{item-iipointmapswap}
$\fdiip(C,p_{-1},p_{1};z) = f(-z)$. 
\item \label{item-iipointmapstereo}
$\fdiip(C_3,e_1,e_{-1})$
with $C_3 = \{x \in \sphere^2:\langle e_3,x\rangle \geq 0\}$
is stereographic projection though $e_{-3}$. 
\end{enumerate}
\end{lemma}

When $C$ is an open region, 
let $\fdiip(C,p_1,p_{-1}) = \fdiip(\overline{C},p_1,p_{-1})$ 
where $\overline{C}$ is the closure of $C$. 

Since area is used in the definition of the map $\fdiip$, we will need the following lemma that the area of the region bounded by a null area curve varies continuously.

\begin{lemma}\label{lemma-area-convergence}
For $k \in \{1,\dots,\infty\}$, let $C_k$ be a closed region in $\sphere^2$ bounded by a simple closed curve.  If $C_k \to C_\infty$ in boundary Fréchet distance, and the boundary of $C_\infty$ has null area, then the area of $C_k$ converges to the area of $C_\infty$. 
\end{lemma}

\begin{warning}
Although $\fdiip$ is defined when the boundary of $C$ has positive area, $\fdiip$ is only continuous over the space of regions with null area boundary.  A crucial issue here is that the conclusion of Lemma \ref{lemma-area-convergence} does not hold in some cases where the boundary of $C_\infty$ has positive area, such as for the region bounded by an Osgood curve \cite{osgood1903jordan}.
\end{warning}

\begin{proof}[Proof of Lemma \ref{lemma-area-convergence}]
We claim that if $\dist(C_k,C_\infty) < \eps$ for $\eps$ sufficiently small, then the symmetric difference $\Delta_k = (C_k \setminus C_\infty)\cup(C_\infty \setminus C_k)$ is contained in $\partial C_\infty \oplus \eps$.
We may assume $\eps$ is small enough that there exists points $p_0,p_1$ such that 
$(p_0 \oplus \eps) \cap C_\infty = \emptyset$ and
$(p_1 \oplus \eps) \subset C_\infty$.
Since $\dist(C_k,C_\infty) < \eps$ there is a map $f : \partial C_\infty \to \partial C_k$ such that $\|f(x)-x\| < \eps$, so $x$ is always closer to $f(x)$ than to $p_0$ or $p_1$.
Therefore, there is a homotopy from $\partial C_k$ to $\partial C_k$ that stays within $\sphere^2\setminus\{p_0,p_1\}$, which implies that $p_0 \not\in C_k$ and $p_1 \in C_k$.
Since this holds for any such pair $p_0,p_1$, the claim holds.

Consider $\eps_1 > 0$.
Since $\partial C_\infty$ has area 0, there is an open cover $U$ of $\partial C_\infty$ that has area at most $\eps_1$.
Since $\sphere^2\setminus U$ and $\partial C_\infty$ are compact and disjoint, they are bounded apart by some $\eps_2 > 0$ depending on $\eps_1$. 
Since $C_k \to C_\infty$, we have for all $k$ sufficiently large that $\dist(C_k,C_\infty) < \eps_2$, 
so by the claim above, $\Delta_k \subseteq \partial C_\infty \oplus \eps_2$, 
so $\Delta_k \subseteq U$, so $\Delta_k$ has area at most $\eps_1$.
Thus, $\mathrm{area}(\Delta_k) \to 0$ as $k\to \infty$, which implies that $\mathrm{area}(C_k) \to \mathrm{area}(C_\infty)$.
\end{proof}

\begin{proof}[Proof of Lemma \ref{lemma-iipointmap}]
To make the problem simpler, we first use a linear fractional transformation $\fswivel$ that sends the unit disk $\mathbf{D}$ to the upper half-plane $\mathbf{H}$.  Such a transformation is given by 
\[ \fswivel(z) = -\im \frac{z-1}{z+1}.\]
We also have $\fswivel(1) = 0$, $\fswivel(-1) = \infty$, and $\fswivel(\im\mathbf{H}\cap\mathbf{D}) = (\mathbf{D}\cap\mathbf{H})$.

By Lemma \ref{lemma-ivpointmap}, there exists an internally conformal homeomorphism $g$ from $\mathbf{D}$ to $C$ such that $g(1) = p_1$ and $g(-1)=p_{-1}$.
Fix $g$, and let $R(x) = [g \circ \fswivel^{-1}](x\mathbf{D}\cap\mathbf{H})$,
and let $a(x)$ be the area of $R(x)$, and let $a_\infty$ be the area of $C$.
Then, $a$ is a strictly increasing function from $a(0) = 0$ to $a(\infty) = a_\infty$, so there is a unique $r>0$ such that $a(r) = \tfrac{a_\infty}{2}$.
Let 
\[ f = g \circ \fswivel^{-1} \circ r\fswivel. \]
Observe that $\fswivel(\im\mathbf{H})=\mathbf{D}$, so 
$f(\im \mathbf{H}\cap\mathbf{D}) = [g \circ \fswivel^{-1}](r\mathbf{D}\cap\mathbf{H}) = R(r)$,
which occupies half the area of $C$.
Hence, $f$ satisfies the defining properties of $\fdiip(C,p_1,p_{-1})$ in the lemma.

To check uniqueness, consider another map $f_0$ satisfying the hypotheses of the lemma.
Then, $h = \fswivel \circ f_0 \circ f^{-1} \circ \fswivel^{-1}$ is a conformal automorphism of the upper half-plane with fixed points $0,\infty$, so $h$ acts by scaling by a positive factor $s$.
If we had $s \neq 1$, then the sets $f_0(\mathfrak{i}\mathbf{H}\cap\mathbf{D})^{\circ}$ and $f(\mathfrak{i}\mathbf{H}\cap\mathbf{D})= R(r)^{\circ}$ would be properly nested open sets, which would contradict that these sets have the same area.  Thus, $h$ is the identity, so $f_0 = f$.

To check part \ref{item-iipointmapcontinuity}, continuity, consider sequences $C_k,p_{1,k},p_{-1,k}$ converging to $C_\infty,p_{1,\infty},p_{-1,\infty}$ under the hypotheses of the lemma.
Let $f_k = \fdiip(C_k,p_{1,k},p_{-1,k})$, and let $R_k = f_k(\im\mathbf{H}\cap\mathbf{D})$. 
Suppose for the sake of contradiction that $f_k(\im)$ does not converge to $f_\infty(\im)$.
Then, there would be some $\eps>0$ such that $|f_k(\im) - f_\infty(\im)| > \eps$, and since $\mathbf{S}^2$ is compact, we could assume $f_k(\im)$ converges to a point $p_2 \neq f_\infty(\im)$; otherwise restrict to a subsequence bounded away from $f_\infty(\im)$ and then restrict to a convergent subsequence. 

We may also assume that $f_k(-\im)$ converges to a point $p_{-2}$.
Let $\widetilde f = \fdivp(p_{\infty,1},p_{\infty,-1},p_2,p_{-2})$ and $\widetilde R = \widetilde f(\im\mathbf{H}\cap\mathbf{D})$.
By uniqueness in Lemma \ref{lemma-ivpointmap}, we have 
$f_k = \fdivp(p_{k,1},p_{k,-1},f_k(\im),f_k(-\im))$,
so by continuity in Lemma \ref{lemma-ivpointmap}, we have 
$f_k \to \widetilde f$ uniformly.
Hence, $R_k \to \widetilde R$ in boundary Fréchet distance. 
Since the boundaries of $\widetilde R$ and $C_\infty$ have null area, by Lemma \ref{lemma-area-convergence}, we have $\area(R_k) \to \area(\widetilde R)$ and $\area(C_k) = 2\area(R_k) \to \area(C_\infty) = 2\area(R_\infty)$.
Hence, $\area(\widetilde R) = \area(R_\infty)$, so by the uniqueness of $\fdiip$, we have $\widetilde f = f_\infty$, so $p_2 = f_\infty(\im)$, which contradicts our choice of subsequence.
Thus, $f_k(\im)$ must converge to $f_\infty(\im)$.
Similarly, $f_k(-\im)$ must converge to $f_\infty(-\im)$.
Therefore, by Lemma \ref{lemma-ivpointmap}, $f_k$ converges to $f_\infty$ uniformly, so part \ref{item-iipointmapcontinuity} holds.

The map $ Q f$ is an internally conformal homeomorphism from $\mathbf{D}$ to $Q C$ that sends $1,-1$ respectively to $Q p_1, Q p_{-1}$, and $Qf(\im\mathbf{H}\cap\mathbf{D})$ occupies half the area of $QC$. 
Since $\fdiip(Q(S,\dots,p_{-2}))$ uniquely satisfies these conditions, part \ref{item-iipointmapsorth} holds.

The map $h = [z \mapsto -\fdiip(C,p_1,p_{-1};\overline{z})]$ is an internally conformal homeomorphism from $\mathbf{D}$ to $-C$ that sends $1,-1$ respectively to $-p_{1}$ and $-p_{-1}$ and $h(\im \mathbf{H}\cap\mathbf{D})$ occupies half the area of $-C$.
Since $\fdiip(-(S,p_1,p_{-2}))$ uniquely satisfies these conditions,
part \ref{item-iipointmapnegate} holds.

The map $h = [z \mapsto \fdiip(C,p_1,p_{-1};-z)]$ is an internally conformal homeomorphism from $\mathbf{D}$ to $C$ that sends $1,-1$ respectively to $p_{-1}$ and $p_1$ and $h(\im \mathbf{H}\cap\mathbf{D})$ occupies half the area of $C$. 
Since $\fdiip(C,p_{-1},p_{1}))$ uniquely satisfies these conditions,
part \ref{item-iipointmapswap} holds.

Finally, stereographic projection though the point $e_{-3}$ from $\mathbf{D}$ to $\sphere^2$ sends 
$\mathbf{D},1,-1$ respectively to $C_3,e_1,e_{-1}$ and sends $\im \mathbf{H}\cap\mathbf{D}$ to a lune occupying half the area of $C_3$, 
so part \ref{item-iipointmapstereo} holds.  
\end{proof}

\subsection{The sphere - $\fsivp$ and  $\fsiip$}

Here we construct two self-homeomorphisms of the sphere, $\fsivp$ and  $\fsiip$.
The map $\fsivp$ depends continuously on a curve and 4 distinct points on the curve, 
while $\fsiip$ depends continuously on a null area curve and 2 distinct points on the curve. 
In each case, the equator is mapped to the given curve.
These maps are constructed by parameterizing the region on either side of the given curve respectively using $\fdivp$ and $\fdiip$, 
and then modifying the parameterizations so that they agree along their common boundary, the curve itself.
Since this is done the same way in both cases, we construct a map $\fstitch$ that modifies a pair parameterizations in such a way.

Given embeddings $f_i : \mathbf{D} \to \sphere^2$ with $p_1 = f_1(1)=f_2(1)$,  $p_{-1} = f_1(-1)=f_2(-1)$,  and $S = f_1(\sphere^1) = f_2(\sphere^1)$,
let 
\[f = \fstitch(f_1,f_2) : \overline{\mb{C}} \to \sphere^2\]
be the homeomorphism defined as follows.
We will assume that both $f_1$ and $f_2$ are orientation preserving;
otherwise replace $f_i$ with $f_i \circ [z \mapsto \overline{z}]$. 
We first define a function $f_0$, which will be the restriction of $f$ to the unit circle in the complex plane.
We will define $f_0$ in terms of its inverse. 
Let 
\[ 
f_0^{-1}: S \to \mathbf{S}^1, 
\quad f_0^{-1}(x) 
= \sqrt{\frac{f_1^{-1}(x)}{f_2^{-1}(x)}}  
\] 
where we choose the branch of $\sqrt{\cdot}$ so that $f_0^{-1}(p_1) = 1$, and define $f_0^{-1}$ from $(f_0^{-1})^2$ on the rest of $S$ by analytic continuation. 
Note that the arc $A$ of $S$ extending in the counter-clockwise direction from $p_1$ around the boundary of $C_1$ as viewed from outside the sphere is sent to the upper half-plane of $\mb{C}$.  Note also that this arc $A$ extends in the clockwise direction from $p$ around the boundary of $C_2$, since $C_2$ is the cell on the opposite side of the sphere from $C_1$. Let 
\[
\fstitch(f_1,f_2;z) = \begin{cases}
f_1(0) & z = 0 \\ 
f_1 \circ |z|f_1^{-1} \circ f_0(\nicefrac{z}{|z|}) & 0 < |z| < 1 \\
f_0(z) & |z| = 1 \\
f_2 \circ \frac{1}{|z|}f_2^{-1} \circ f_0(\nicefrac{z}{|z|}) & 1 < |z| < \infty \\
f_2(0) & z = \infty \\ 
\end{cases}
\]

Note that we use the metric on $\overline{\mb{C}}$ induced by stereographic projection, which makes $\overline{\mb{C}}$ homeomorphic to the 2-sphere.

\begin{lemma}\label{lemma-stitch}
The map $\fstitch$ satisfies the following.
\begin{enumerate}
\item \label{item-stitchcontinuity}
In the sup-metric, $\fstitch(f_1,f_2)$ depends continuously on $f_1,f_2$ in the sup-metric.
\item \label{item-stitchsorth}
For $Q \in \sorth_3$, $\fstitch(Q(f_1,f_2)) = Q\fstitch(f_1,f_2)$.
\item \label{item-stitchnegate}
$\fstitch(-(f_1,f_2);z) = -\fstitch(f_1,f_2;\overline{z})$.
\item \label{item-stitchswap}
$\fstitch(f_2,f_1;z) = \fstitch(f_1,f_2;\nicefrac{1}{z})$.
\end{enumerate}
\end{lemma}


\begin{proof}
Consider embeddings $f_{i,k}$ of $\mathbf{D}$ in $\sphere^2$ such that $(f_{1,k},f_{2,k})$ is in the domain of $\fstitch$ for $k \in \{1,\dots,\infty\}$,  
and such that $f_{i,k} \to f_{i,\infty}$ uniformly. 
By Lemma \ref{lemma-invariable-convergence}, 
$f_{i,k} \to f_{i,\infty}$ invariably, so $f_{i,k}^{-1} \to f_{i,\infty}^{-1}$ invariably.

Let $f_{0,k}$ be as in the definition of $\fstitch$ for the two maps $f_{1,k},f_{2,k}$. 
The branch $f_{\mathrm{sqrt}}$ of the multifunction $[(x_1,x_2) \mapsto \sqrt{\nicefrac{x_1}{x_2}}]$ on $(\mathbf{S}^1 \cap \mathbf{H})\times (\mathbf{S}^1 \cap -\mathbf{H})$ that sends $(1,1)$ to $1$ is continuous, and so is the branch on $(\mathbf{S}^1 \cap -\mathbf{H})\times (\mathbf{S}^1 \cap \mathbf{H})$.
Also, for $x \in S_k$, $f_{1,k}^{-1}(x)$ and $f_{2,k}^{-1}(x)$ always either coincide at $1$ or $-1$, or are on opposite sides of the real line, so the branch $f_{\mathrm{sqrt}}$ is continuous on the range of $f_{1,k}^{-1}\times f_{2,k}^{-1}$ over $S_k$.  Hence, by Lemma \ref{lemma-invariable-convergence},
$f_{0,k}^{-1} \to f_{0,\infty}^{-1}$ invariably, so $f_{0,k} \to f_{0,\infty}$, and $f_{1,k}^{-1} \circ f_{0,k} \to f_{1,\infty}^{-1} \circ f_{0,\infty}$, and $f_{2,k}^{-1} \circ f_{0,k} \to f_{2,\infty}^{-1} \circ f_{0,\infty}$ invariably as well.

Let $g_{1,k}(z) = |z|f_1^{-1} \circ f_0(\nicefrac{z}{|z|})$ for $z \in \mathbf{D}\setminus 0$ and $g_{1,k}(0) = 0$, 
and define $g_{2,k}$ analogously.
We claim that $g_{i,k} \to g_{i,\infty}$ invariably. 
Consider $\eps > 0$, and let $K(\eps) > 0$ be such that for all $k \geq K(\eps)$ and all $z \in \mathbf{D}$, we have $|[f_{1,k}^{-1} \circ f_{0,k}](z) - [f_{1,\infty}^{-1} \circ f_{0,\infty}](z)| < \eps$ as in the definition of uniform convergence.
Then, we have 
$|g_{1,k}(z)-g_{1,\infty}(z)| < \tfrac{2|z|}{1+|z|^2} \eps < \eps $
in the metric on $\overline{\mb{C}}$ induced by stereographic projection 
for all $k \geq K(\eps)$, so $g_{1,k} \to g_{1,\infty}$ uniformly.
Therefore, by Lemma \ref{lemma-invariable-convergence}, $g_{1,k} \to g_{1,\infty}$ invariably, and similarly 
$g_{2,k} \to g_{2,\infty}$ invariably as claimed,
which implies that
$\fstitch(f_{1,k},f_{2,k}) \to \fstitch(f_{1,\infty},f_{2,\infty})$ invariably.  
Thus, $\fstitch(f_{1,k},f_{2,k}) \to \fstitch(f_{1,\infty},f_{2,\infty})$ uniformly,
so part \ref{item-stitchcontinuity} holds.




Let $f_0$ and $\widetilde f_0$ be as in the definition of $\fstitch$ for the pairs $(f_1,f_2)$ and $Q(f_1,f_2)$.
The map $[Qf_i]^{-1}$ sends $Qx$ to $f_i^{-1}(x)$, so 
on $QS$, 
the map $\widetilde f_0{}^{-1}$
sends $Qx$ to the branch of 
$\sqrt{f_1^{-1}(x)/f_2^{-1}(x)}$ that sends $Qp_1$ to $1$.
Hence, $\widetilde f_0 = Qf_0$.
In the case where $0 < |z| < 1$, we have 
$\fstitch(Q(f_1,f_2);z) 
= Q f_1 |z|f_1^{-1} Q^{-1} \widetilde f_0(z) 
= Q f_1 |z|f_1^{-1} f_0(z) 
= Q\fstitch(f_1,f_2;z) $,
and similarly in the other cases.
Therefore, $\fstitch(Q(f_1,f_2)) = Q\fstitch(f_1,f_2)$,
so part \ref{item-stitchsorth} holds.

Now let $f_0$ and $\widetilde f_0$ be as in the definition of $\fstitch$ for the pairs $(f_1,f_2)$ and $-(f_1,f_2)$.
Since the antipodal map $[x \mapsto -x]$ reverses orientation on the sphere, 
the map 
$\widetilde f_0{}^{-1}$ sends $-x$ to the branch of 
$\sqrt{\overline{f_1^{-1}(x)}/\overline{f_2^{-1}(x)}} = \overline{\sqrt{f_1^{-1}(x)/f_2^{-1}(x)}}$ that sends $-p_1$ to $1$.
Hence, $\widetilde f_0(z) = -f_0(\overline{z})$.
In the case where $0 < |z| < 1$, we have 
$\fstitch(-(f_1,f_2);z) 
= [- f_1 |z|f_1^{-1} \circ - \widetilde f_0](z) 
= - f_1 |z|f_1^{-1} f_0(\overline{z}) 
= -\fstitch(f_1,f_2;\overline{z}) $,
and similarly in the other cases.
Therefore, $\fstitch(-(f_1,f_2);z) = -\fstitch(f_1,f_2;\overline{z})$,
so part \ref{item-stitchnegate} holds.

For $x \in S$, we have
\begin{align*}
[\fstitch(f_2,f_1)]^{-1}(x)  
&= \sqrt{\frac{f_2^{-1}(x)}{f_1^{-1}(x)}} \\
&= \frac{1}{\sqrt{\frac{f_1^{-1}(x)}{f_2^{-1}(x)}}} \\
&= [[z\mapsto\nicefrac{1}{z}]\circ [\fstitch(f_1,f_2)]^{-1}](x) \\ 
&= [\fstitch(f_1,f_2)\circ [z\mapsto\nicefrac{1}{z}]]^{-1}(x),
\end{align*}
where the branch of $\sqrt{\cdot}$ is always chosen so that $p_1$ is sent to $1$. 
Thus, $\fstitch(f_2,f_1;z) = \fstitch(f_1,f_2;\nicefrac1z)$ on $z \in \sphere^1$, which means 
part \ref{item-stitchswap} holds on $z \in \sphere^1$.
Consider $z \in \overline{\mb{C}}$, and let $x = \fstitch(f_2,f_1;\nicefrac{z}{|z|}) \in S$. 
In the case where $0 < |z| < 1$, we have $1 < |\nicefrac{1}{z}| < \infty$, so
$ 
\fstitch(f_1,f_2;\nicefrac{1}{z}) = 
f_2 \circ |z|f_2^{-1}(x) = \fstitch(f_2,f_1;z)
$.
Also, $\fstitch(f_1,f_2;\nicefrac{1}{0}) = f_2(0) = \fstitch(f_2,f_1;0)$,
so part \ref{item-stitchswap} holds on $\mathbf{D}$, and similarly in the case where $1 < |z| \leq \infty$, 
so part \ref{item-stitchswap} holds on all of $\overline{\mb{C}}$.
\end{proof}

For a simple closed curve $S \subset \sphere^2$ and 4 distinct points $p_1,p_2,p_{-1},p_{-2}$ appearing in that order around the curve, 
let
\[ \fsivp(S,p_1,p_2,p_{-1},p_{-2}) 
= \fstitch(\fdivp(S,p_1,p_2,p_{-1},p_{-2}),\fdivp(S,p_1,p_{-2},p_{-1},p_{2})) \]

\begin{lemma} \label{lemma-sivpointmap}
The map $f = \fsivp(S,p_1,p_2,p_{-1},p_{-2})$ satisfies the following.
\begin{enumerate}
\item \label{item-sivpointmapcontinuity}
$f$ in the sup-metric depends continuously on $S$ in Fréchet distance and on $p_1,\dots, p_{-2}$.
\item \label{item-sivpointmaporth}
For $Q \in \orth_3$, 
$ \fsivp(Q(S,p_1,p_2,p_{-1},p_{-2}))
= Qf$.
\item \label{item-sivpointmapswap}
$\fsivp(S,p_{-1},p_{-2},p_1,p_2;z) 
= f(-z)$.
\item \label{item-sivpointmapstereo}
$\fsivp(e_3^\bot,e_1,e_2,e_{-1},e_{-2})$ 
is stereographic projection though $e_{-3}$. 
\end{enumerate}
\end{lemma}

\begin{proof}
Since $f_1 =\fdivp(S,p_1,p_2,p_{-1},p_{-2})$ and $f_2 = \fdivp(S,p_1,p_{-2},p_{-1},p_{2})$ in the sup-metric depend continuously on $S,p_1,p_2,p_{-1},p_{-2}$ by Lemma \ref{lemma-ivpointmap}, and $f$ in the sup-metric depends continuously on $f_1$ and $f_2$, 
part \ref{item-sivpointmapcontinuity} holds. 

By Lemmas \ref{lemma-ivpointmap} and \ref{lemma-stitch}, 
we have the following. 
In the case where $Q \in \sorth_3$ we have 
\begin{align*}
\fsivp (Q(S,p_1,p_2,p_{-1},p_{-2})) 
&= \fstitch(\fdivp(Q(S,p_1,p_2,p_{-1},p_{-2})),\fdivp(Q(S,p_1,p_{-2},p_{-1},p_{2}))) \\
&= \fstitch(Q\fdivp(S,p_1,p_2,p_{-1},p_{-2}),Q\fdivp(S,p_1,p_{-2},p_{-1},p_{2})) \\
&= Q\fstitch(\fdivp(S,p_1,p_2,p_{-1},p_{-2}),\fdivp(S,p_1,p_{-2},p_{-1},p_{2})) \\
&= Qf.
\end{align*}
In the case where $Q \in -\sorth_3$ we have 
\begin{align*}
\fsivp (Q(S,p_1,p_2,p_{-1},p_{-2});z) 
&= \fstitch(\fdivp(Q(S,p_1,p_2,p_{-1},p_{-2})),\fdivp(Q(S,p_1,p_{-2},p_{-1},p_{2}));z) \\
&= \fstitch(Q\fdivp(S,p_1,p_2,p_{-1},p_{-2}),Q\fdivp(S,p_1,p_{-2},p_{-1},p_{2});\overline{z}) \\
&= Q\fstitch(\fdivp(S,p_1,p_2,p_{-1},p_{-2}),\fdivp(S,p_1,p_{-2},p_{-1},p_{2});z) \\
&= Qf.
\end{align*}
Hence, part \ref{item-sivpointmaporth} holds.
\begin{align*}
\fsivp(S,p_{-1},p_{-2},p_1,p_2;z)
&= \fstitch(\fdivp(S,p_{-1},p_{-2},p_{1},p_{2}),\fdivp(S,p_{-1},p_{2},p_{1},p_{-2};z)) \\ 
&= \fstitch(\fdivp(S,p_{1},p_{2},p_{-1},p_{-2}),\fdivp(S,p_{1},p_{-2},p_{-1},p_{2};-z)) \\ 
&= f(-z).
\end{align*}
Hence, part \ref{item-sivpointmapswap} holds.

The map $f_1 = \fdivp(e_3^\bot,e_1,e_2,e_{-1},e_{-2})$ is stereographic projection though $e_{-3}$,
and by part \ref{item-sivpointmaporth}, the map $f_{2} = \fdivp(e_3^\bot,e_1,e_{-2},e_{-1},e_{2})$ is stereographic projection though $e_{3}$, so the map $[z \mapsto f_{2}(\nicefrac1z)]$ is the stereographic projection on $\overline{\mb{C}\setminus\mathbf{D}}$ though $e_{-3}$.
Therefore, the map $\sqrt{f_1^{-1}/f_{2}^{-1}}$ is the stereographic projection on $\sphere^1$ though $e_{-3}$.
Thus, $\fsivp(e_3^\bot,e_1,e_2,e_{-1},e_{-2}) = \fstitch(f_1,f_2)$ is stereographic projection though $e_{-3}$, so part \ref{item-sivpointmapstereo} holds. 
\end{proof}

For a region $C$ bounded by a null area simple closed curve, and distinct points $p_1,p_{-1}$ on the curve, 
let
\[ \fsiip(C,p_1,p_{-1}) 
= \fstitch(\fdiip(C,p_1,p_{-1}),\fdiip({\sphere^2\setminus C},p_1,p_{-1})) \]

\begin{lemma} \label{lemma-siipointmap}
The map $f = \fsiip(C,p_1,p_{-1})$ satisfies the following.
\begin{enumerate}
\item \label{item-siipointmapcontinuity}
$f$ in the sup-metric depends continuously on $C$ in boundary Fréchet distance and on $p_1,p_{-1}$.
\item \label{item-siipointmaporth}
For $Q \in \orth_3$, 
$ \fsiip(Q(C,p_1,p_{-1}))
= Qf$.
\item \label{item-siipointmapswap}
$\fsiip({\sphere^2\setminus C},p_{-1},p_1;z) 
= f(-\nicefrac1z)$.
\item \label{item-siipointmapstereo}
$\fsiip(C_3,e_1,e_{-1})$ 
with $C_3 = \{x \in \sphere^2: \langle e_3, x\rangle \geq 0\}$ 
is stereographic projection though $e_{-3}$. 
\end{enumerate}
\end{lemma}

\begin{proof}
Since $f_1 =\fdiip(C,p_1,p_{-1})$ and $f_2 = \fdiip({\sphere^2\setminus C},p_1,p_{-1}))$ in the sup-metric depend continuously on $C,p_1,p_{-1}$ by Lemma \ref{lemma-iipointmap}, and $f$ in the sup-metric depends continuously on $f_1$ and $f_2$ by Lemma \ref{lemma-stitch}, 
part \ref{item-sivpointmapcontinuity} holds. 
\begin{align*}
\fsiip({\sphere^2\setminus C},p_{-1},p_1;z)
&= \fstitch(\fdiip({\sphere^2\setminus C},p_{-1},p_{1}),\fdiip(C,p_{-1},p_{1});z) \\
&= \fstitch(\fdiip({\sphere^2\setminus C},p_{1},p_{-1}),\fdiip(C,p_{1},p_{-1});-z) \\
&= \fstitch(\fdiip(C,p_{1},p_{-1}),\fdiip({\sphere^2\setminus C},p_{1},p_{-1}));-\nicefrac1z) \\
&= f(-\nicefrac1z)
\end{align*}
Hence, part \ref{item-siipointmapswap} holds.
Parts \ref{item-siipointmaporth} and \ref{item-siipointmapstereo} hold by the same argument as in the proof of Lemma~\ref{lemma-sivpointmap}.
\end{proof}

\section{The deformation}\label{sectionDeformation}

\begin{figure}
\includegraphics[scale=0.9]{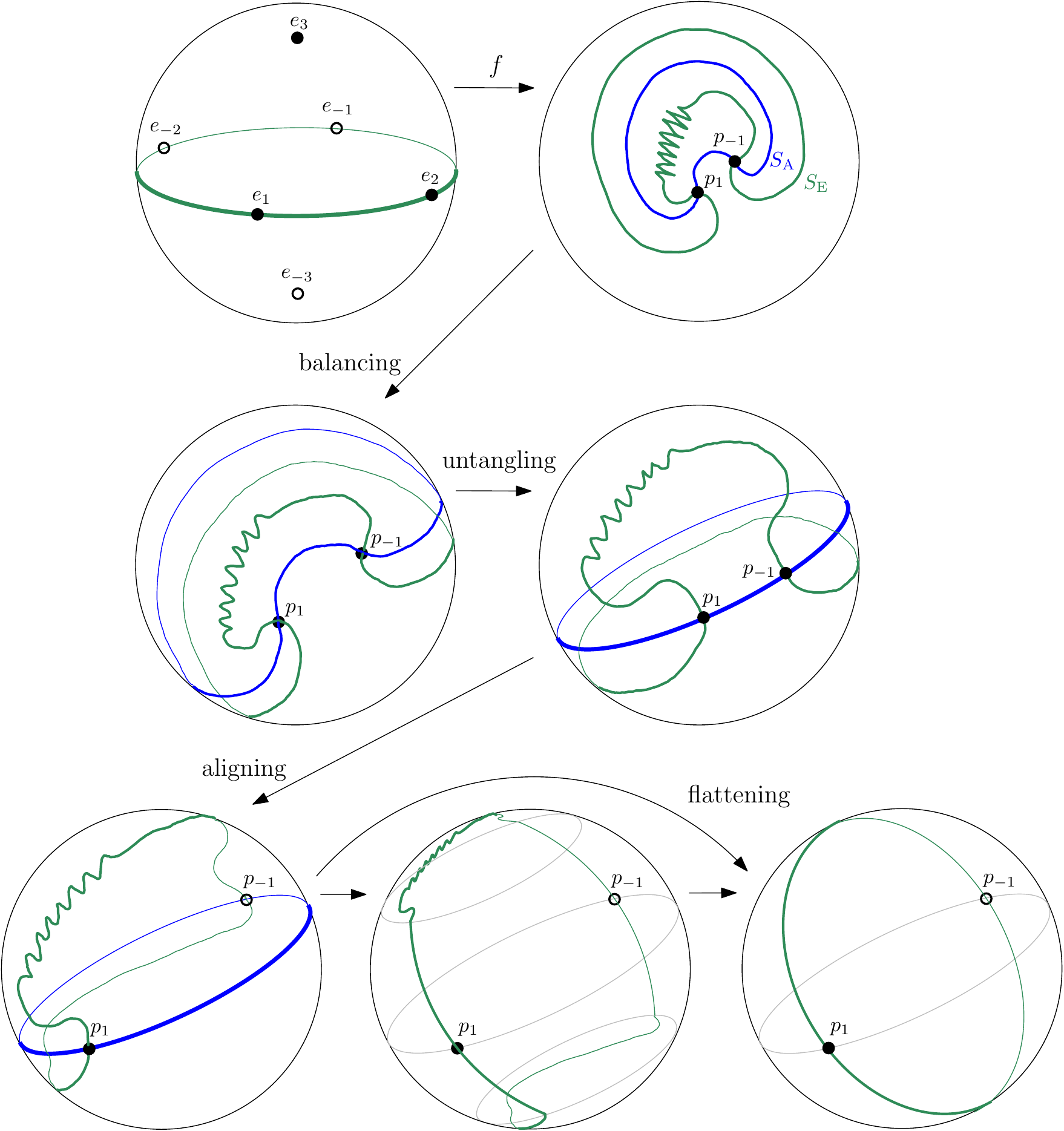}
\caption{%
First 4 stages of the deformation retraction starting from a map $f$.
}
\label{figureFirst}
\end{figure}

\begin{figure}
\includegraphics[scale=0.9]{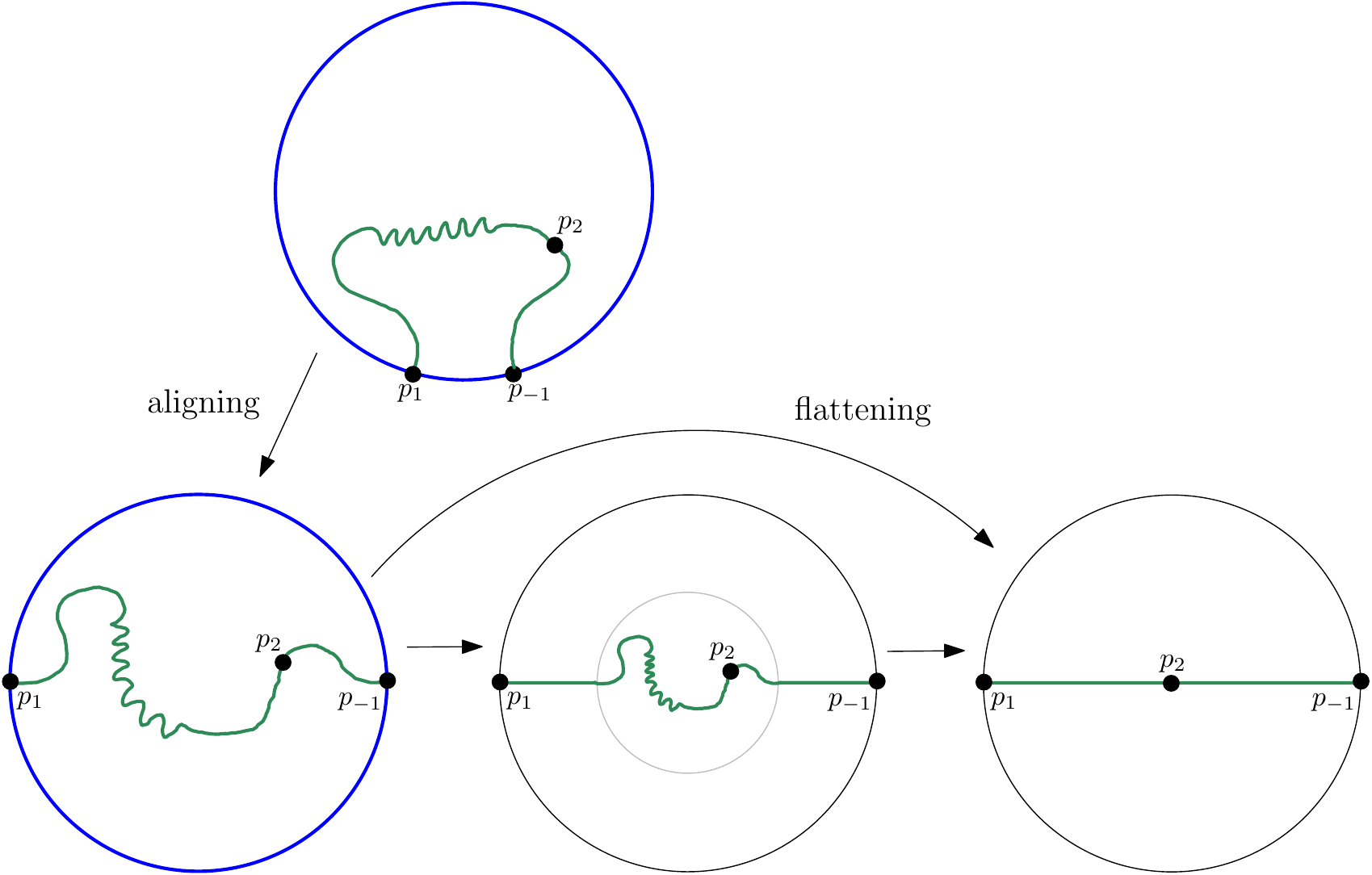}
\caption{%
Stages 3 and 4 of the deformation retraction.
}
\label{figureThreeFour}

\includegraphics[scale=0.9]{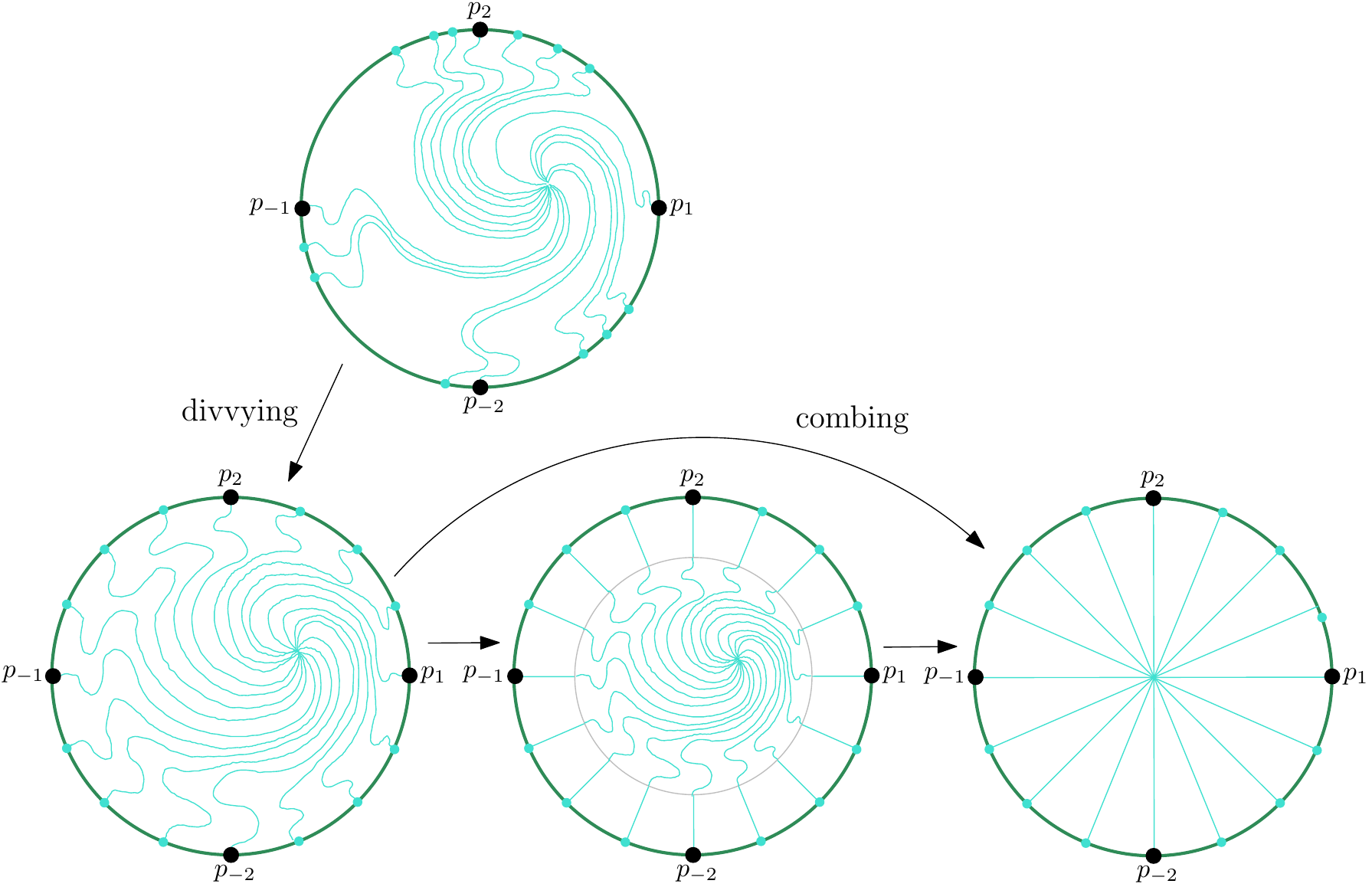}
\caption{%
Stages 5 and 6 of the deformation retraction.
The thin lines in the disk show how evenly spaced meridians are mapped.
}
\label{figureFiveSix}
\end{figure}

Our goal here is to construct the deformation retraction implied by Theorem \ref{theoremDeformation}, which we denote by $\rho$.  
We first define a collection of objects that we will refer to when we construct $\rho$.
These objects are defined in terms of $\rho$ or in terms of each other. 
In each stage of the deformation, we define the evolution of some of these objects, and $\rho$ and the other objects will then be determined as a consequence of the defining relationships between the objects.

Recall $e_{-i} = -e_i$ and $e_i$ denotes the $i$th standard basis vector.
For $f \in \hom(\sphere^2)$ and $t \in [0,6]_\mb{R}$, let
\begin{align*}
p_i(f,t) &= \rho(f,t;e_i) \quad \text{for } i \in \{-3,-2,-1,1,2,3\}, \\ 
\smagellan(f,t) &= \rho(f,t;e_3^\bot), \\
\fmagellan(f,t) &= \fsivp\left([\smagellan,p_1,p_2,p_{-1},p_{-2}](f,t)\right). \\
\samundsen(f,0) &= \fmagellan\left(f,t;\overline{\mb{R}}\right) \\
&= \bigcup_{s \in \{1,-1\}} \fdivp\left([\smagellan,p_1,p_{s2},p_{-1},p_{-s2}](f,0);[-1,1]_{\mb{R}}\right), \\
\samundsen(f,t) &= [ \rho(f,t)\circ f^{-1}] (\samundsen(f,0)),  \\
\intertext{
let $C_1(f,t)$ be the region bounded by $\samundsen(f,t)$ that contains the point $p_2$, and $C_{-1}(f,t)$ be that containing $p_{-2}$, and let }
\famundsen(f,t) &= \fsiip([C_1,p_1,p_{-1}](f,t)). 
\end{align*}
The letter E in $\smagellan$ and $\fmagellan$ is for equator, and A in $\samundsen$ and $\famundsen$ is for Roald Amundsen, a polar explorer. 
We refer to $e_3^\bot$ as the equator.
In some cases we may suppress the argument for cleaner notation.
For example, we may simply write $p_1$ for $p_1(f,t)$, as long as the meaning can be understood from context. 

The deformation will proceed through 6 stages: balancing, untangling, aligning, flattening, divvying, and combing.  
Each stage will occur over a unit interval, so that the whole deformation occurs over $[0,6]_\mb{R}$ instead of a single unit interval.
The first 4 stages deform the image of the equator, $\smagellan$, to a great circle; see Figures \ref{figureFirst} and \ref{figureThreeFour}.
The first stage, balancing, deforms the curve $\samundsen$ to a bisector of the sphere.
That is, $\samundsen$ is deformed to make the regions on either side each have area $2\pi$. 
The second stage, untangling, deforms $\samundsen$ to a great circle.
The third stage, aligning, moves the points $p_1$ and $p_{-1}$ into antipodal position.
The fourth stage, flattening, deforms the curve $\smagellan$ to a great circle.
Once the equator is mapped to a great circle, we deal the map in the two hemispheres on either side separately; see Figure \ref{figureFiveSix}. 
The fifth stage, divvying, deforms the map to an isometry on the equator, and the sixth stage, combing, deforms the map to an isometry on the rest of the sphere.

\subsection{Balancing}

In the first stage, we deform $\samundsen$ so that each of the regions on either side have area $2\pi$. 

Let $\omega(f,t) = \famundsen(f,0)\circ \e^t\famundsen(f,0)^{-1} $.
Let $a(f,t)$ be the area of 
\[D(f,t) = \omega(f,t;C_1(f,0)) =  \famundsen(f,0;\e^{t}\mathbf{D}).\]
Observe that 
$a(f)$ is a strictly increasing continuous function on $\overline{\mb{R}}$ from $a(-\infty)=0$ to $a(\infty) = 4\pi$.
Let $T(f) \in \mb{R}$ such that $a(f,T(f))= 2\pi$.
For $t \in [0,1]_\mb{R}$, let 
\[ \rho(f,t) = \omega(f,tT(f)) \circ f.\]

\begin{lemma}\label{lemmarectifiable}
For all $f \in \hom(\sphere^2)$ and $t \in \mb{R}$, 
the boundary of $D(f,t)$ is rectifiable, and so has null area. 
Also, $C_1(f,1) = D(f,T(f))$ and has area $2\pi$. 
\end{lemma}

\begin{lemma}\label{lemmabalancing}
For all $t \in [0,1]_\mb{R}$ and $f \in \hom(\sphere^2)$, we have the following. 
\begin{enumerate}
\item \label{itemcontinuous-a}
In the sup-metric, $\rho(f,t)$ depends continuously on $f$ and $t$.
\item \label{itemequivariant-a}
For all $Q \in \orth_3$, and $s \in \{1,-1\}$, 
we have $\rho(Qfs,t) = Q\rho(f,t)s$. 
\item \label{itemstrong-a}
If $f \in \orth_3$, then $\rho(f,t) = f$.
\item \label{itemstart}
$\rho(f,0) = f$.
\end{enumerate}
\end{lemma}

\begin{proof}[Proof of Lemma \ref{lemmarectifiable}]
In the case where $t=0$, the boundary of $D(f,t)$ is $\samundsen(f,0)$, which is the union of two smooth curves, so the boundary is rectifiable.

Consider the case where $t < 0$.
Let 
$f_1 = \fdiip([C_1,p_1,p_{-1}](f,0))$ and $f_0 = \famundsen(f,0)$ restricted to $\sphere^1$.
Then, 
\begin{align*}
D(f,t)
&= \famundsen(f,0;\e^t\mathbf{D}) \\
&= \fsiip([C,p_1,p_{-1}](f,0);\e^t\mathbf{D}) \\
&= [z \mapsto [f_1\circ |z|f_1^{-1}\circ f_0](\nicefrac{z}{|z|})](\e^t\mathbf{D}\setminus 0) \cup \{0\} \\
&= f_1(\e^{t}\mathbf{D}),
\end{align*}
since $f_1^{-1}\circ f_0$ sends $\sphere^1$ to $\sphere^1$ bijectively. 
Since $f_1$ is internally conformal, the boundary of $D(f,t)$ is smooth and therefore rectifiable. 
The case where $t>0$ follows similarly. 
Also, 
\begin{align*}
C_1(f,1) 
&= [ \rho(f,1)\circ f^{-1}](C_1(f,0)) \\
&= [\omega(f,T(f))] (C_1(f,0)) \\
&= [\famundsen(f,0)\circ \e^{T(f)}\famundsen(f,0)^{-1}](C_1(f,0)) \\
&= \famundsen(f,0; \e^{T(f)} \mathbf{D}) \\
&= D(f,T(f)),
\end{align*}
which by definition of $T$ has area $2\pi$. 
\end{proof}

\begin{proof}[Proof of Lemma \ref{lemmabalancing} part \ref{itemcontinuous-a}]
Consider $f_k\in\hom(\sphere^2)$ and $t_k \in [0,1]_\mb{R}$ in the domain of $\rho$ such that $f_k \to f_\infty$ uniformly and $t_k \to t_\infty$.
Then, $\smagellan(f_k,0) \to \smagellan(f_\infty,0)$ in Fréchet distance $p_i(f_k,0) \to p_i(f_\infty,0)$, 
so $\samundsen(f_k,0) \to \samundsen(f_\infty,0)$ in Fréchet distance by Lemma \ref{lemma-sivpointmap}, which implies that $C(f_k,0) \to C(f_\infty,0)$ in boundary Fréchet distance, so $\famundsen(f_k,0) \to \famundsen(f_\infty,0)$ uniformly,
which implies that $\famundsen(f_k,0)^{-1} \to \famundsen(f_\infty,0)^{-1}$ uniformly as well. 

Suppose for the sake of contradiction that $T(f_k)$ does not converge to $T(f_\infty)$.
Since $\overline{\mb{R}}$ is compact, we may assume that $T(f_k) \to \widetilde T \neq T(f_\infty)$ for some $\widetilde T \in \overline{\mb{R}}$.
Let us assume for now that $\widetilde T < T(f_\infty)$. 
Then, $\e^{T(f_\infty)}\mathbf{D} \setminus \e^{\widetilde T}\mathbf{D}$ has non-empty interior, 
so $D(f_\infty,T(f_\infty)) \setminus D(f_\infty,\widetilde T) = \famundsen(f_\infty,0;\e^{T(f_\infty)}\mathbf{D} \setminus \e^{\widetilde T}\mathbf{D})$ also has non-empty interior, which implies that 
$a(f_\infty,\widetilde T) < a(f_\infty,T(f_\infty)) = 2\pi$. 
Also, $D(f_k,T(f_k)) \to D(f_\infty,\widetilde T)$ in boundary Fréchet distance, 
and by Lemma \ref{lemmarectifiable}, the boundary of $D(f_\infty,\widetilde T)$ is rectifiable, so by Lemma \ref{lemma-area-convergence}, 
$a(f_k,T(f_k)) \to a(f_\infty,\widetilde T) < 2\pi$,
so  
$a(f_k,T(f_k)) < 2\pi$ for some $k$ sufficiently large,
but $a(f_k,T(f_k)) = 2\pi$ by definition of $T$, 
which is a contradiction.
Similarly, we have $a(f_k,T(f_k)) > 2\pi$ for some $k$ sufficiently large in the case where $\widetilde T > T(f_\infty)$, which also contradicts the definition of $T$.
Hence, $T(f_k)$ must converge to $T(f_\infty)$.

Thus, $\omega(f_k,t_kT(f_k)) \to \omega(f_\infty,t_\infty T(f_\infty))$ uniformly, 
and therefore $\rho(f_k,t_k) \to \rho(f_\infty,t_\infty)$ uniformly. 
\end{proof}


\begin{proof}[Proof of Lemma \ref{lemmabalancing} part \ref{itemequivariant-a}]
\[ \smagellan(Qfs,0) = Qfs(e_3^\bot) = Qf(e_3^\bot) = Q\smagellan(f,0). \]
By Lemma \ref{lemma-sivpointmap}, we have 
\begin{align*} 
\samundsen(Qfs,0) 
&= \fsivp\left([\smagellan,p_1,p_2,p_{-1},p_{-2}](Qfs,0);\overline{\mb{R}}\right) \\
&= \fsivp\left(Q[\smagellan,p_s,p_{s2},p_{-s},p_{-s2}](f,0);\overline{\mb{R}}\right) \\
&= Q\fsivp\left([\smagellan,p_1,p_2,p_{-1},p_{-2}](f,0);[z \mapsto sz](\overline{\mb{R}})\right) \\
&= Q \samundsen(f,0).
\end{align*}
Since $\samundsen(Qfs,0) = Q \samundsen(f,0)$ and $p_2(Qfs,0) = Qp_{s2}(f,0)$, 
we have $C_1(Qfs,0) = QC_s(f,0)$.
By Lemma \ref{lemma-siipointmap}, we have 
\begin{align} 
\famundsen(Qfs,0) 
&= \fsiip([C_1,p_1,p_{-1}](Qfs,0)) \nonumber \\
&= \fsiip(Q[C_s,p_{s},p_{-s}](f,0)) \nonumber \\
&= Q\famundsen(f,0)\circ[z \mapsto sz^s], \label{equation-famundsen-equivariant}
\end{align}
\begin{align*} 
\omega(Qfs,t) 
&= \famundsen(Qfs,0)\circ \e^t \famundsen(Qfs,0)^{-1} \\
&= Q\famundsen(f,0)\circ s^2\e^{st} \famundsen(f,0)^{-1}Q^{-1} \\
&= Q\omega(f,st)Q^{-1},
\end{align*} 
so $D(Qfs,t) = QD(f,st)$,
so $a(Qfs,t) = a(f,st)$,
so $T(Qfs) = sT(f)$, so 
\begin{align*}
\rho(Qfs,t) 
&= \omega(Qfs,tT(Qfs)) Qfs \\
&= Q\omega(f,stsT(f))Q^{-1} Qfs \\
&= Q\rho(f,t)s. \qedhere
\end{align*}

\end{proof}


\begin{proof}[Proof of Lemma \ref{lemmabalancing} parts \ref{itemstrong-a} and \ref{itemstart}]
For part \ref{itemstart}, we have 
\[ 
\rho(f,0) 
= \omega(f,0) \circ f 
= \famundsen(f,0)\circ \e^0 \famundsen(f,0)^{-1}\circ f 
= f.
\]
Next, consider the case $f=\id$.
In this case, $\smagellan(\id,0) = e_3^\bot$ and $p_i(\id,0) = e_i$, so 
by Lemma \ref{lemma-sivpointmap}, 
the map $\fmagellan(\id,0)$ is stereographic projection though $e_{-3}$, 
so $\samundsen(\id,0) = e_2^\bot$ and $C(\id,0) = \{x\in\sphere^2: \langle e_2,x\rangle\geq 0\}$, which has area $2\pi$, so $T(\id) = 0$.

In the case $f \in \orth_3$, we have $T(f) = T(\id) = 0$, so 
\[
\rho(f,t) 
= \omega(f,tT(f)) \circ f 
= \omega(f,0) \circ f 
= f. \qedhere
\]
\end{proof}

\subsection{Untangling}\label{subsectionUntangling}

In the second stage, we deform $\samundsen$ to a great circle. 
To do so, $\samundsen$ will evolve by level-set flow, which is a weak formulation of curvature flow for Jordan curves of area 0, an extremely mild smoothness condition. 
Here are the important properties of level-set flow for our purpose.
Joseph Lauer showed that level-set flow is defined for all null area simple closed curves up to some positive time depending on the curve, and that the curve immediately becomes smooth and evolves by curvature flow  \cite{lauer2016evolution}.
Michael Gage showed that 
for bisectors, i.e., null area simple closed curves that divide the sphere into 2 regions that each have area $2\pi$, 
the evolution by curvature flow approach a great circle as time becomes infinite \cite{gage1990curve}.
Dobbins showed that for bisectors, level-set flow depends continuously in Fréchet distance on initial conditions \cite{dobbins2021continuous}. 
Sigurd Angenet showed that for a pair of smooth curves evolving by curvature flow, the number of intersection points is finite and non-increasing, provided that the curves are initially distinct \cite{angenent1991parabolic}.
Dobbins showed that the points of intersection move along a trajectory that depends continuously on initial conditions \cite{dobbins2021continuous}.

As $\samundsen$ evolves, we will also have to determine how the points $p_1$ and $p_{-1}$ move. 
For this we introduce another curve $\sgagarin$, which will also evolve by curvature flow, and we will let $p_1$ and $p_{-1}$ be the points where the curves $\samundsen$ and $\sgagarin$ intersect.
By Angenet's theorem this is well defines for positive time, but this might not be well defined at the limit as $t$ becomes infinite. 
To deal with this we will determine a time $T_1$, after which the points $p_1$ and $p_{-1}$ will move according to reparameterizations of the regions on either side of $\samundsen$ using the Carathéodory mapping theorem.




Let $L(f) = L_1\cup L_{-1} \subset \overline{\mb{C}}$ be the lune consisting of 2 circular arcs $L_1,L_{-1}$ from $1$ to $-1$ such that $\famundsen(f,1;L_n)$ bisects $C_n$. 
Let $\gamma_\mathrm{A}(f)$ and $\gamma_\mathrm{L}(f)$ be the level-set flows starting from 
$\gamma_\mathrm{A}(f;0) = \samundsen(f,1)$ and $\gamma_\mathrm{L}(f;0) = \famundsen(f,1;L(f))$.
Note that since $\gamma_\mathrm{A}(f;0)$ is a bisector, $\gamma_\mathrm{A}(f;\infty) = \lim_{t \to \infty} \gamma_\mathrm{A}(f;t)$ exists and is a great circle,
and likewise for $\gamma_\mathrm{L}$. 

Let $\tau(t) = \frac{t-1}{2-t}$.
Note that $\tau$ is strictly increasing on $[1,2]_\mb{R}$ and sends $[1,2]_\mb{R}$ to $[0,\infty]_\mb{R}$.
For $t \in (1,2]_\mb{R}$, let 
$\samundsen(f,t) = \gamma_\mathrm{A}(f,\tau(f,t))$
and $\sgagarin(f,t) = \gamma_\mathrm{L}(f,\tau(f,t))$. 
Let $C_n$ be on the same side of $\samundsen$ throughout the evolution. 

Let $U = U(f) = \{u_1,u_{-1}\} \subset \sphere^2$ be the two points on the sphere that are perpendicular to the plane spanned by $\samundsen(f,2)$ with $u_1 \in C_1(f,2)$.  Note that $\samundsen(f,2)$ is a great circle, so $U$ is well defined. 

For $(f,t)$ and $n$ such that $u_n(f) \in (C_n(f,t))^\circ$, 
let 
$\mr{h}_n(f,t) : C_n(f,2) \to C_n(f,t)$ be the internally conformal homeomorphism that is fixed at $u_n = h_n(f,t;u_n)$ and $\mathrm{D}(h_n(f,t);u_n) = \lambda\id$ with $\lambda > 0$ where $\mathrm{D}(h,x)$ denotes the total derivative of a function $h$ at $x$. 
Note that $h_n$ is uniquely determined by these conditions by the Schwarz lemma. 
Let 
\begin{align*}
r_\mr{h}(f,t) 
&= \sup\left\{ \left\| h_n(f,t;x) - x\right\| 
: x \in \samundsen(f,2), n \in \{1,-1\} \right\}, \\ 
\intertext{where $r_\mr{h}$ is only defined when $h_1$ and $h_{-1}$ are both defined, }
r_\mr{u}(f,t) 
&= \sup\left\{\left|\left(\|x-u_n\|\right) - \sqrt{2}\right| 
: x \in \samundsen(f,t), n \in \{1,-1\} \right\}, \\
r_\mr{t}(f,t) 
&= \sup\left\{ r_\mr{h}(f,s), r_\mr{u}(f,s) 
: t \leq s \leq 2 \right\}, \\
w(f,t) &= \left(1-\mr{r_t}(f,t)\right)^+
\end{align*}
where $(x)^+ = \max(x,0)$ is the positive part of a function. 
We will show that $r_\mr{t}(f,t)$ is positive and decreasing on $t \in [1,2]_\mb{R}$. 
Let $T_1(f)$ be the weighted median value of $t$ with weight $w(f,t)$, i.e.\ such that 
\[ \int_{t=1}^{T_1(f)} w(f,t) \mathrm{d} t = \frac12 \int_{t=1}^{2} w(f,t) \mathrm{d} t. \]

For a pair of bisectors that intersect at a pair of points, there is a unique trajectory starting at each intersection point that remains in the intersection of the curves as they evolve by level-set flow \cite{dobbins2021continuous}. 
For $t \leq T_1(f)$, let 
$p_1(f,t)$ be the trajectory starting from $p_1(f,1)$ that remains in $[\samundsen \cap \sgagarin](f,t)$,
and define $p_{-1}$ analogously. 
For $t > T_1(f)$, let
\[p_n(f,t) = [h_n(f,t)\circ h_n(f,T_1(f))^{-1}](p_n(f,T_1(f))). \]

For $t \in (1,2]_\mb{R}$, 
let 
\[\rho(f,t) = \famundsen(f,t) \circ \famundsen(f,1)^{-1} \circ \rho(f,1). \]

\begin{lemma}\label{lemmasamundsen}
For all $f \in \hom(\sphere^2)$, $\samundsen(f,2)$ is a great circle. 
\end{lemma}

\begin{lemma}\label{lemmauntangling}
For all $t \in [1,2]_\mb{R}$ and $f \in \hom(\sphere^2)$, we have the following. 
\begin{enumerate}
\item \label{itemcontinuous-b}
In the sup-metric, $\rho(f,t)$ depends continuously on $f$ and $t$.
\item \label{itemequivariant-b}
For all $Q \in \orth_3$, and $s \in \{1,-1\}$, 
we have $\rho(Qfs,t) = Q\rho(f,t)s$. 
\item \label{itemstrong-b}
If $f \in \orth_3$, then $\rho(f,t) = f$.
\end{enumerate}
\end{lemma}

\begin{proof}[Proof of Lemma \ref{lemmasamundsen}]
Curves evolving by curvature flow approach a great circle \cite{gage1990curve}, 
and $\samundsen(f,2)$ is defined as the limit as $t \to \infty$ of such a curve $\gamma_\mr{A}(f,t)$, so $\samundsen(f,2)$ is a great circle. 
\end{proof}

\begin{proof}[Proof of Lemma \ref{lemmauntangling} part \ref{itemcontinuous-b}]

Consider $f_k \to f_\infty$ in the sup-metric and $t_k \to t_\infty$.
By Lemma \ref{lemmabalancing}, $\samundsen(f_k,1) \to \samundsen(f_\infty,1)$ in Fréchet distance and $p_n(f_k,1) \to p_n(f_\infty,1)$, 
so by Lemma \ref{lemma-siipointmap} $\famundsen(f_k,1) \to \famundsen(f_\infty,1)$ in the sup-metric.

Suppose for the sake of contradiction that $L_1(f_k)$ does not converge to $L_1(f_\infty)$ in Fréchet distance.
We can parameterize the space of circular arcs through $\mathbf{D}$ from $1$ to $-1$ by the point where the arc intersects the segment $[-\im,\im]_\mb{C}$.  
Moreover, the space of such arcs is compact, provided we include the two semicircles on $\sphere^1$ between $1$ and $-1$.
Hence, we may assume that $L_1(f_k) \to \widetilde L \neq L_1(f_\infty)$; otherwise restrict to a convergent subsequence.
Let $\sgagarin_{1,k} = \famundsen(f_k,1;L_1(f_k))$,
and $\widetilde S = \famundsen(f_k,1;\widetilde L)$. 
Then, $\sgagarin_{1,k} \to \widetilde S$ in Fréchet distance, 
since $\famundsen(f_\infty,1;L_1(f_k)) \to \widetilde S$ and the Fréchet distance between $\famundsen(f_\infty,1;L_1(f_k))$ and $\sgagarin_{1,k}$ converges to 0. 
The curves $\samundsen(f_\infty,1)$ and $\widetilde S$ are piecewise smooth, and therefore rectifiable, 
so by Lemma \ref{lemma-area-convergence}, the areas of the regions of $C_1(f_k)$ on either side of $\sgagarin_{1,k}$ converge to those of $\widetilde S$,  
but $\widetilde S$ is on one side of $\sgagarin_{1,\infty}$,
so $\widetilde S$ does not bisect the sphere. 
Hence, for some $k$ sufficiently large, $\sgagarin_{1,k}$ does not bisect the sphere, 
but that contradicts the definition of $L_1$.
Thus, $L_1(f_k)$ must converge to $L_1(f_\infty)$ in Fréchet distance.
Similarly, $L_{-1}(f_k)$ converges to $L_{-1}(f_\infty)$,
so $L(f_k)$ converges to $L(f_\infty)$ in Fréchet distance.

With this, we have that $\sgagarin(f_k,1) = \famundsen(f_k,1;L(f_k)) \to \sgagarin(f_\infty,1)$ in Fréchet distance, since $L(f_k) \to L(f_\infty)$ and $\famundsen(f_k,1) \to \famundsen(f_\infty,1)$, 
and $\samundsen(f_k,1) \to \samundsen(f_\infty,1)$. 
Since level-set flow depends continuously on initial conditions for bisectors \cite{dobbins2021continuous}, $\sgagarin(f_k,t_k) \to \sgagarin(f_\infty,t_\infty)$ and $\samundsen(f_k,t_k) \to \samundsen(f_\infty,t_\infty)$ in Fréchet distance.

By Angenet's theorem  
$[\samundsen \cap \sgagarin](f,t)$ is a pair of points for $t \in [1,2)_\mb{R}$ \cite{angenent1991parabolic}, 
and the author showed that these points move along a trajectory that depends continuously on $\samundsen(f,1)$ and $\sgagarin(f,1)$ in Fréchet distance \cite{dobbins2021continuous}. 
Let $ p_{\mr{AL},n}(f,t) \in [\samundsen \cap \sgagarin](f,t)$ be the point on the trajectory starting from $p_{\mr{AL},n}(f,1)$ for $n \in \{1,-1\}$. 
By \cite[Lemma 3.1.3]{}, 
$p_{\mr{AL},n}(f_k,t_k) \to p_{\mr{AL},n}(f_\infty,t_\infty)$ provided that $t_\infty < 2$.

Since $\samundsen(f_k,t_k) \to \samundsen(f_\infty,t_\infty)$ 
and $u_n(f_k) \to u_n(f_\infty)$,   
we have $r_\mr{u}(f_k,t_k) \to r_\mr{u}(f_\infty,t_\infty)$.

In the case where $u_n(f_\infty,t_\infty) \in (C_n(f_\infty,t_\infty))^\circ$,
we have that $u_n(f_\infty,t_\infty)$ is bounded away from $\sphere^2\setminus (C_n(f_\infty,t_\infty))^\circ$,
so for all $k$ sufficiently large, $u_n(f_k,t_k) \in (C_n(f_k,t_k))^\circ$.
Hence $h_n(f_k,t_k)$ is defined for all $k$ sufficiently large, and by Rado's theorem, 
$h_n(f_k,t_k) \to h_n(f_\infty,t_\infty)$ uniformly, 
so $r_\mr{h}(f_k,t_k) \to r_\mr{h}(f_\infty,t_\infty)$.

Let $T_0$ be the last time that there is a point of $U(f_\infty)$ on $\samundsen(f_\infty,T_0)$.
Consider the case where $t_\infty > T_0$.
Consider $\eps > 0$. 

By the definition of $r_\mr{t}$, 
there is $s \geq t_\infty$ such that 
$r_c(f_\infty,s) > r_\mr{t}(f_\infty,t_\infty) - \eps/2$ 
where $c$ is either `$\mr{h}$' or `$\mr{u}$'.
In either case, we have already shown that $r_c(f_k,s) \to r_c(f_\infty,s)$.
In the case where $s = t_\infty$, 
we have $r_c(f_k,t_k) \to r_c(f_\infty,s)$,
so for all $k$ sufficiently large, $r_\mr{t}(f_k,t_k) \geq r_c(f_k,t_k) > r_c(f_\infty,s) - \eps/2$.
In the case where $s > t_\infty$, 
for all $k$ sufficiently large, $r_c(f_k,s) > r_c(f_\infty,s) - \eps/2$ 
and $t_k < s$, so again $r_\mr{t}(f_k,t_k) > r_c(f_\infty,s) - \eps/2$.
In any case $r_\mr{t}(f_k,t_k) > r_\mr{t}(f_\infty,t_\infty) - \eps$.

Since $r_c(f,t)$ is continuous as a function of $f$ and $t$, 
for all $k$ sufficiently large and all $s$ such that $|s-t_\infty| \leq |t_k - t_\infty|$, we have that  
$r_c(f_k,s) < r_c(f_\infty,t_\infty) +\eps \leq r_\mr{t}(f_\infty,t_\infty) +\eps $.
We also have by definition of $r_\mr{t}$ that for all $s \geq t_\infty$ 
that $r_c(f_\infty,s) < r_\mr{t}(f_\infty,t_\infty) + \eps/2$ 
for both $c= \mr{h}$ and $c = \mr{u}$.  
For all $k$ sufficiently large, we have $r_c(f_k,s) < r_c(f_\infty,s) +\eps/2$, so $r_c(f_k,s) < r_\mr{t}(f_\infty,t_\infty) + \eps$.
Together, we now have that for all $k$ sufficiently large, all $s \geq t_k$, and $c \in \{\rm{h,u}\}$ that $r_c(f_k,s) < r_\mr{t}(f_\infty,t_\infty) +\eps$.
Therefore, $r_\mr{t}(f_k,t_k) \leq r_\mr{t}(f_\infty,t_\infty) + \eps$.

We now have for all $k$ sufficiently large that $|r_\mr{t}(f_k,t_k) - r_\mr{t}(f_\infty,t_\infty)| \leq \eps$
Hence, $r_\mr{t}(f_k,t_k) \to r_\mr{t}(f_\infty,t_\infty)$, provided $t_\infty > T_0$. 

In the case where $t_\infty = T_0$, we have $r_\mr{u}(f_k,t_k) \to r_\mr{u}(f_\infty,t_\infty) = \sqrt{2}$, which is the largest value $r_\mr{h}$ or $r_\mr{u}$ can attain, so $r_\mr{t}(f_k,t_k) \to r_\mr{t}(f_\infty,t_\infty) = \sqrt{2}$ in this case as well, and likewise in the case where $t_\infty < T_0$, since for $k$ sufficiently large, $\sqrt{2} \geq r_\mr{t}(f_k,t_k) \geq r_\mr{u}(f_k,T_0) \to \sqrt{2} = r_\mr{t}(f_\infty,T_0)$.
Thus, $r_\mr{t}(f,t)$  is continuous as a function of $f$ and $t$, and therefore $w$ is continuous as well.

As $t \to 2$, we have $h_n(f,t) \to \id_{C_n(f,2)}$, since $\samundsen(f,t) \to \samundsen(f,2)$, so $r_\mr{h}(f,t) \to 0$, and likewise for $r_\mr(u)(f,t)$, since $\samundsen(f,t)$ converges to the great circle orthogonal to $u_1$ and $u_{-1}$, so $r_\mr{t}(f,t)$ is decreasing and converges to $0$ as $t \to 2$.
Hence, $w(f,t)$ is non-negative and increasing as a function of $t \in [1,2]_\mb{R}$ and converges to 1 as $t \to 2$.
Moreover, $w(f)$ is positive on its support, which is the half open interval $I(f) = (T_\mr{w}(f),2]_\mb{R}$ where $T_\mr{w}(f)$ is the last time that $r_\mr{t}(f,T_\mr{w}(f)) = 1$, or $I(f) = [1,2]_\mb{R}$ in the case where $r_\mr{t}(f,t) < 1$ for all $t$.

We now have that 
\[ W(f,T) = \int_{t=1}^{T} w(f,t) \mathrm{d} t \]
is continuous as a function of $T$ and strictly increasing on $I(f)$, which is the support of $W(f)$.
Therefore, $T_1(f)$ is well defined.
By Lebesgue's dominated convergence theorem $W(f_k,T) \to W(f_\infty,T)$.
Since $w(f) < 1$, $W(f)$ is 1-Lipschitz continuous for all $f$,
and since $W(f_k) \to W(f_\infty)$ pointwise, by the Arzelà-Ascoli theorem, $W(f_k)$ converges uniformly.
Therefore, $[W(f_k)]^{-1}$ converges uniformly, so 
$T_1(f_k) = [W(f_k)]^{-1}\left(\tfrac12 W(f_k,2)\right) \to T_1(f_\infty)$. 

Since $C_n(f_k,t_k) \to C_n(f_\infty,t_\infty)$ and $u_n(f_k,t_k) \to u_n(f_\infty,t_\infty)$, Rado's theorem implies that 
$h_n(f_k,t_k) \to h_n(f_\infty,t_\infty)$ uniformly.
This also implies that $[h_n(f_k,T(f_k))]^{-1} \to [h_n(f_\infty,T(f_\infty))]^{-1}$ invariably.

We now show continuity of $p_n$. 
First, consider the case where $t_k \leq T_1(f_k)$ for all $k$ sufficiently large. 
Then, $p_n(f_k,t_k) = p_{\mr{AL},n}(f_k,t_k)$, so $p_n(f_k,t_k) \to p_n(f_\infty,t_\infty)$. 
This also means that $p_n(f_k,T(f_k)) \to p_n(f_\infty,T(f_\infty))$.  
Next, consider the case where $t_k \leq T_1(f_k)$ for all $k$ sufficiently large. 
Then, 
\[p_n(f_k,t_k) = [h_n(f_k,t_k)\circ [h_n(f_k,T_1(f_k))]^{-1}](p_n(f_k,T_1(f_k))), \]
and $p_n(f_k,T_1(f_k)) \to p_n(f_\infty,T_1(f_\infty))$ and both $h_n(f_k,t_k)$ and $[h_n(f_k,T_1(f_k))]^{-1}$ converge invariably, 
so again we have $p_n(f_k,t_k) \to p_n(f_\infty,t_\infty)$. 
If neither of the above cases hold, then we can partition the sequence $p_n(f_k,t_k)$ depending on how $t_k$ compares to $T_1(f_k)$, and since $p_n(f_k,t_k)$ converges to $p_n(f_\infty,t_\infty)$ for each part, we have $p_n(f_k,t_k) \to p_n(f_\infty,t_\infty)$ in all cases.

Thus, $\famundsen(f_k,t_k) \to \famundsen(f_\infty,t_\infty)$ by Lemma \ref{lemma-siipointmap} and therefore $\rho(f_k,t_k) \to \rho(f_\infty,t_\infty)$ in the sup-metric.
\end{proof}

\begin{proof}[Proof of Lemma \ref{lemmauntangling} part \ref{itemequivariant-b}]
By Lemma \ref{lemmabalancing}, we have $\samundsen(Qfs,1) = Q\samundsen(f,1) $ and $C_n(Qfs,1) = QC_{sn}(f,1)$.
Hence, By Lemma \ref{lemma-siipointmap}, we have $\famundsen(Qfs,1) = Q\famundsen(f,1) \circ [z \mapsto sz^s] $ as in Equation \ref{equation-famundsen-equivariant}, 
so $L(Qfs) = sL(f)^s$,
so $\gamma_\mr{L}(Qfs;0) = Q\gamma_\mr{L}(f;0)$ and $\gamma_\mr{A}(Qfs;0) = Q\gamma_\mr{A}(f;0)$.
Since the curvature $\kappa\nu$ of a curve commutes with isometries of the sphere, we have 
for $t \in [0,\infty]_\mb{R}$ that 
$\gamma_\mr{L}(Qfs;t) = Q\gamma_\mr{L}(f;t)$ and $\gamma_\mr{A}(Qfs;t) = Q\gamma_\mr{A}(f;t)$, so 
for $t \in [1,2]_\mb{R}$ we have that 
$\samundsen(Qfs,t) = Q\samundsen(f,t)$ and $\sgagarin(Qfs,t) = Q\sgagarin(f,t)$.
Hence, $C_n(Qfs,t) = QC_{sn}(f,t)$ and $u_n(Qfs) = Qu_{sn}(f)$.

In the case where $Q \in \sorth_3$, both $Q$ and $Q^{-1}$ are conformal, 
so $Qh_{sn}(f,t)Q^{-1}$ is a conformal map. 
In the case where $Q \in -\sorth_3$, both $Q$ and $Q^{-1}$ are angle preserving and orientation reversing, 
so again $Qh_{sn}(f,t)Q^{-1}$ is a conformal map. 
In both cases, we have a conformal map  
from $C_{n}(Qfs,2)$ to $C_{n}(Qfs,t)$ that is fixed at $u_n(Qfs)$.
By the chain rule, 
\begin{align*}
\mr{D}(Qh_{sn}(f,t)Q^{-1};u_n(Qfs)) 
&= \mr{D}(Q;h_{sn}(f,t)(u_{sn}(f))) \circ \mr{D}(h_{sn}(f,t);u_{sn}(f)) \circ \mr{D}(Q^{-1};Qu_{sn}(f)) \\
&= Q \circ \lambda \id \circ Q^{-1} \\
&= \lambda \id
\end{align*}
where $\lambda >0$ is as given in the definition of $h_{sn}$.
Hence, for all case of $Q \in \orth_3$ we have $h(Qfs,t) = Qh_{sn}(f,t)Q^{-1}$. 
Therefore, $r_\mr{h}(Qfs,t) = r_\mr{h}(f,t)$ and $r_\mr{u}(Qfs,t) = r_\mr{u}(f,t)$, so $w(Qfs,t) = w(f,t)$, which implies that $T_1(Qfs) = T_1(f)$. 

By Lemma \ref{lemmabalancing}, we also have $p_n(Qfs,1) = Qp_{sn}(f,1)$ for $n \in \{1,-1\}$,
so $p_n(Qfs,t) = Qp_{sn}(f,t)$ in the case where $t \in [1,T_1(f)]_\mb{R}$ since 
$\samundsen(Qfs,t) = Q\samundsen(f,t)$ and $\sgagarin(Qfs,t) = Q\sgagarin(f,t)$,
and likewise in the case where $t \in [T_1(f),2]_\mb{R}$ since 
\begin{align*}
p_n(Qfs,t) 
&= [h_n(Qfs,t) [h_n(Qfs,T_1(Qfs))]^{-1}](p_n(Qfs,T_1(Qfs))). \\
&= [Qh_{sn}(f,t)Q^{-1}Q[h_{sn}(f,T_1(f))]^{-1}Q^{-1}](Qp_{sn}(f,T_1(f))). \\
&= Q[h_{sn}(f,t)[h_{sn}(f,T_1(f))]^{-1}](p_{sn}(f,T_1(f))). \\
&= Qp_{sn}(f,t)
\end{align*}

Therefore, by lemma \ref{lemma-siipointmap}, 
$\famundsen(Qfs,t) = Q\famundsen(f,t)\circ [z\mapsto sz^s]$, so 
\begin{align*}
\rho(Qfs,t) 
&= \famundsen(Qfs,t) \circ \famundsen(Qfs,1)^{-1} \circ \rho(Qfs,1) \\
&= Q\famundsen(f,t) \circ [z\mapsto sz^s] \circ [z\mapsto sz^s]^{-1} \circ [\famundsen(f,1)]^{-1} Q^{-1} \circ Q \rho(f,1)s \\
&= Q\famundsen(f,t) \circ \famundsen(f,1)^{-1} \circ \rho(f,1)s \\
&= Q\rho(f,t)s  \qedhere
\end{align*}
\end{proof}

\begin{proof}[Proof of Lemma \ref{lemmauntangling} part \ref{itemstrong-b}]
Consider $f \in \orth_3$. By Lemma \ref{lemmabalancing}, we have $\rho(f,1) = f$.
Hence, $\samundsen(f,1) = \samundsen(f,0)$ is a great circle and $p_1(f,1) = p_1(f,0)$ and $p_{-1}(f,-1) = p_{-1}(f,0)$ are antipodes, 
so by Lemma \ref{lemma-iipointmap}, $\famundsen(f,1)$ is isometric to a stereographic projection, so $\sgagarin(f,1)$ is also a great circle.
Since great circles have no geodesic curvature, the evolution of a great circle by curvature flow is trivial, so 
$\samundsen(f,t)=\samundsen(f,1)$ and $\sgagarin(f,t)=\sgagarin(f,1)$, so 
$u_n(f,t) = p_{n2}(f,1)$ and $h_n(f,t) = \id$, so 
$r_\mr{h}(f,t) = 0$ and $r_\mr{u}(f,t) = 0$, so 
$w(f,t) = 1$, so $T_1(f) = \nicefrac32$.
Hence, $p_n(f,t) = p_n(f,1)$ for $n \in \{1,-1\}$; 
in the case of $t \leq \nicefrac32$ because $\samundsen(f,t)$ and $\sgagarin(f,t)$ remain constant in $t$, and in the case of $t > \nicefrac32$ because $h_n(f,t) = \id$. 
Therefore, $\famundsen(f,t) = \famundsen(f,1)$,
which implies that $\rho(f,t) = \rho(f,1) = f$.
\end{proof}

\subsection{Aligning}

In the third stage, we move the points $p_1$ and $p_{-1}$ to antipodal positions. 

For $t \in (2,3]_\mb{R}$, we define the following.
Let $\samundsen(f,t) = \samundsen(f,2)$. 
Let $p_1(f,t)$ move at uniform speed along the arc from $p_1(f,2)$ to $-p_{-1}(f,2)$ that avoids $-p_1(f,2)$ and reach the midpoint of this arc at $t=3$,
and analogously, 
let $p_{-1}(f,t)$ move at uniform speed along the arc from $p_{-1}(f,2)$ to $-p_{1}(f,2)$ that avoids $-p_{-1}(f,2)$ and reach the midpoint of this arc at $t=3$.  Let 
\[\rho(f,t) = \famundsen(f,t) \circ \famundsen(f,2)^{-1} \circ \rho(f,2). \]

\begin{lemma}\label{lemmaaxis}
The points $p_1(f,3),p_{-1}(f,3)$ are antipodes.
\end{lemma}

\begin{lemma}\label{lemmaaligning} 
On $t \in [2,3]_\mb{R}$, the following hold.
\begin{enumerate}
\item \label{itemcontinuous-c}
In the sup-metric, $\rho(f,t)$ depends continuously on $f$ and $t$.
\item \label{itemequivariant-c}
For all $Q \in \orth_3$, and $s \in \{1,-1\}$, 
we have $\rho(Qfs,t) = Q\rho(f,t)s$. 
\item \label{itemstrong-c}
If $f \in \orth_3$, then $\rho(f,t) = f$.
\end{enumerate}
\end{lemma}

\begin{proof}[Proof of Lemma \ref{lemmaaxis}]
In the case where $p_{-1}(f,2) = -p_1(f,2)$, both points $p_{1}(f,t)$ and $p_{-1}(f,t)$ remain constant, so $p_1(f,3),p_{-1}(f,3)$ are antipodes.
Otherwise, the signed angle $\theta$ from $p_1(f,2)$ to $-p_{-1}(f,2)$ is the same as the signed angle from $-p_{1}(f,2)$ to $p_{-1}(f,2)$, and the signed angle from $p_1(f,2)$ to $p_1(f,3)$ as well as that from $-p_1(f,2)$ to $p_{-1}(f,3)$ is $\nicefrac\theta2$, so again the points $p_1(f,3),p_{-1}(f,3)$ are antipodes. 
\end{proof}

\begin{proof}[Proof of Lemma \ref{lemmaaligning} part \ref{itemcontinuous-c}]
Consider $f_k \in \hom(\sphere^2)$ and $t_k \in [2,3]_\mb{R}$ such that $f_k \to f_\infty$ uniformly and $t_k \to t_\infty$.
By Lemma \ref{lemmauntangling}, $\samundsen(f_k,t_k) = \samundsen(f_k,2) \to \samundsen(f_\infty,t_\infty)$ and $C_1(f_k,t_k) \to C_1(f_\infty,t_\infty)$. Also, the point that is a $((t-2)/2)$-fraction of the way along the arc from $p_1$ to $-p_{-1}$ varies continuously as a function of $p_1$, $-p_{-1}$, and $t$.  Hence, $p_1(f_k,t_k) \to p_{1}(f_\infty,t_\infty)$ and likewise for $p_{-1}$, so by Lemma \ref{lemma-siipointmap} $\famundsen(f_k,t_k) \to \famundsen(f_\infty,t_\infty)$ uniformly, so by Lemma \ref{lemma-invariable-convergence}, $[\famundsen(f_k,2)]^{-1} \to [\famundsen(f_\infty,2)]^{-1}$ uniformly, and therefore 
$\rho(f_k,t_k) \to \rho(f_\infty,t_\infty)$. 
\end{proof}

\begin{proof}[Proof of Lemma \ref{lemmaaligning} part \ref{itemequivariant-c}]
By Lemma \ref{lemmauntangling} we have $\samundsen(Qfs,t) = \samundsen(Qfs,2) = Q\samundsen(f,2) = Q\samundsen(f,t)$ and $C_1(Qfs,t) = QC_{s}(f,t)$.  Also, $p_1(Qfs,2) = Qp_s(f,2)$ and $-p_{-1}(Qfs,2) = -Qp_{-s}(f,2)$, so the point $p_1(Qfs,t)$ that is a $((t-2)/2)$-fraction of the way along the arc from $p_1(Qfs,2)$ to $-p_{-1}(Qfs,2)$ is also the image by $Q$ of the point $p_s(f,t)$ that is a $((t-2)/2)$-fraction of the way along the arc from $p_s(f,2)$ to $-p_{-s}(f,2)$, which means that $p_1(Qfs,t) = Qp_s(f,t)$ and likewise for $p_{-1}$. 
Therefore, by Lemma \ref{lemma-siipointmap}, $\famundsen(Qfs,t) = Q\famundsen(f,t) \circ [z \mapsto sz^s]$, 
which implies that $\rho(Qfs,t) = Q\rho(f,t)s$. 
\end{proof}

\begin{proof}[Proof of Lemma \ref{lemmaaligning} part \ref{itemstrong-c}]
Suppose $f \in \orth_3$. 
Then, by Lemma \ref{lemmauntangling}, $\rho(f,2) = f \in \orth_3$, so $p_1(f,2)$ and $p_{-1}(f,2)$ are antipodes, so $p_1(f,t) = p_1(f,2)$ and $p_{-1}(f,t) = p_{-1}(f,2)$ are constant in $t$, so $\famundsen(f,t) = \famundsen(f,2)$ is constant in $t$, so $\rho(f,t) = \rho(f,2) = f$. 
\end{proof}

\subsection{Flattening} 

In the fourth stage, we deform $\smagellan$ to a great circle. 
In each hemisphere, we map $\smagellan$ to a curve $\sorpheus$ in the disk, 
where we deform $\sorpheus$ to a line segment, 
and we let $\smagellan$ be the preimage of $\sorpheus$. 
We deform $\sorpheus$ to a segment by continuously shrinking $\sorpheus$ to the origin and connecting the end points to the boundary of the disk by segments; see Figure \ref{figureThreeFour}.

For $t \in (3,4]_\mb{R}$, we define the following.
Let $p_i(f,t) = p_i(f,3)$ for $i \in \{1,-1\}$ and 
$r(t) = 4-t$. 
Observe $r$ sends $[3,4]_\mb{R}$ to $[1,0]_\mb{R}$. Let 
\begin{align*}
\sorpheus(f,3) &= [\famundsen(f,3)]^{-1}(\smagellan(f,3)), \\
\sorpheus(f,t) &= \bigcup \left\{ \begin{array}{r@{\hspace{12pt}}l@{\vspace{3pt}}}
r(t)(\sorpheus(f,3)\cap\mathbf{D}), &
\frac{1}{r(t)}(\sorpheus(f,3)\setminus\mathbf{D}), \\
\left[\frac{-1}{r(t)},-r(t)\right]_\mb{R} , &
\left[r(t),\frac{1}{r(t)}\right]_\mb{R}  \\
\end{array} \right\}, \\ 
\smagellan(f,t) &= [\famundsen(f,3)]( \sorpheus(f,t) ), \\
p_2(f,t) &= \left[\famundsen(f,3)\circ r(t)[\famundsen(f,3)]^{-1}\right](p_2(f,3)), \\
p_{-2}(f,t) &= \left[\famundsen(f,3) \circ \tfrac{1}{r(t)}[\famundsen(f,3)]^{-1}\right](p_{-2}(f,3)), \\
\rho(f,t) &= \fmagellan(f,t) \circ \fmagellan(f,3)^{-1} \circ \rho(f,3). 
\end{align*}

\begin{lemma}\label{lemmasmagellan}
The curve $\smagellan(f,4)$ is a great circle,
and $\rho(f,t)$ is well defined for $t \in (3,4]$. 
\end{lemma}

\vbox{
\begin{lemma}\label{lemmaflattening} 
On $t \in [3,4]_\mb{R}$, the following hold.
\begin{enumerate}
\item \label{itemcontinuous-d}
In the sup-metric, $\rho(f,t)$ depends continuously on $f$ and $t$.
\item \label{itemequivariant-d}
For all $Q \in \orth_3$, and $s \in \{1,-1\}$, 
we have $\rho(Qfs,t) = Q\rho(f,t)s$. 
\item \label{itemstrong-d}
If $f \in \orth_3$, then $\rho(f,t) = f$.
\end{enumerate}
\end{lemma}
}

\begin{proof}[Proof of Lemma \ref{lemmasmagellan}]
Since $\smagellan(f,3)$ and $\samundsen(f,3)$ intersect in exactly the pair of points $\{p_1,p_{-1}\}(f,3)$, the curve $\sorpheus(f,3)$ consists of a pair of curves from $1$ to $-1$; one going through $\mathbf{D}$ and the other going through $\overline{\mb{C}\setminus\mathbf{D}}$.
Therefore, $\sorpheus(f,t)$ consists of a curve from $-r(t)$ to $r(t)$ through $r(t)\mathbf{D}$, and a segment from $r(t)$ to $\frac1{r(t)}$, and a curve from $\frac1{r(t)}$ to $\frac{-1}{r(t)}$ through $\overline{\mb{C}\setminus\frac1{r(t)}\mathbf{D}}$, and a segment from $\frac{-1}{r(t)}$ back to $-r(t)$.
Hence, $\sorpheus(f,t)$ is a simple closed curve in $\overline{\mb{C}}$, so $\smagellan(f,t)$ is a simple closed curve. 

By definition, we have $p_2(f,3) \in \smagellan(f,3) \cap C_1(f,3)^\circ$, 
so $[\famundsen(f,3)]^{-1}(p_2(f,3)) \in \sorpheus(f,3) \cap \mathbf{D}^\circ$,
so $r(t)[\famundsen(f,3)]^{-1}(p_2(f,3)) \in r(t)(\sorpheus(f,t) \cap \mathbf{D}^\circ) \subset \sorpheus(f,t) \cap C_1(f,3)^\circ$,
so $p_2(f,t) \in \smagellan(f,t)\cap C_1(f,3)^\circ$.
Similarly, $p_{-2}(f,t) \in \smagellan(f,t)\cap C_{-1}(f,3)^\circ$.
Hence, the points $p_1,p_2,p_{-1},p_{-2}$ appear in that order around $\smagellan$, so 
$\fmagellan(f,t)$ is well defined, and therefore $\rho(f,t)$ is well defined.

The curve $\samundsen(f,4) = \samundsen(f,2)$ is a great circle by Lemma \ref{lemmasamundsen},
and the points $p_1(f,4) = p_1(f,3)$ and $p_{-1}(f,4) = p_{-1}(f,3)$ are antipodal by Lemma \ref{lemmaaxis}.
Therefore, by Lemma \ref{lemma-siipointmap}, $\famundsen(f,4)$ is isometric to a stereographic projection, 
and by definition $\sorpheus(f,4) = \mb{R}$, so $\smagellan(f,4) = \famundsen(f,4;\sorpheus(f,4))$ is a great circle. 
\end{proof}

\begin{proof}[Proof of Lemma \ref{lemmaflattening} part \ref{itemcontinuous-d}]
Consider $f_k \in \hom(\sphere^2)$ and $t_k \in [3,4]_\mb{R}$ such that $f_k \to f_\infty$ uniformly and $t_k \to t_\infty$.
By Lemma \ref{lemmaaligning}, $\rho(f_k,3) \to \rho(f_\infty,3)$ uniformly, so $p_i(f_k,3) \to p_i(f_\infty,3)$ and $\smagellan(f_k,3) \to \smagellan(f_\infty,3)$ and $\samundsen(f_k,3) \to \samundsen(f_\infty,3)$ in Fréchet distance, so by Lemma \ref{lemma-siipointmap}, $\famundsen(f_k,3) \to \famundsen(f_\infty,3)$ uniformly, so by Lemma \ref{lemma-invariable-convergence}, $[\famundsen(f_k,3)]^{-1} \to [\famundsen(f_\infty,3)]^{-1}$ uniformly, and therefore $\sorpheus(f_k,3) = [\famundsen(f_k,3)]^{-1}(\smagellan(f_k,3)) \to \sorpheus(f_\infty,3)$ in Fréchet distance.

Each part of $\sorpheus$ is composed of continuous functions, so $\sorpheus(f_k,t_k) \to \sorpheus(f_\infty,t_\infty)$ in Fréchet distance, and $\famundsen(f,3)$ is uniformly continuous by the Heine-Cantor theorem, so $\smagellan(f_k,t_k) \to \smagellan(f_\infty,t_\infty)$ in Fréchet distance.
Also, $p_n(f_k,t_k) \to p_n(f_\infty,t_\infty)$ for $n \in \{1,-1,2,-2\}$, so by Lemma \ref{lemma-sivpointmap}, $\fmagellan(f_k,t_k) \to \fmagellan(f_\infty,t_\infty)$ uniformly.
Therefore, $\rho(f_k,t_k) \to \rho(f_\infty,t_\infty)$ by Lemma \ref{lemma-invariable-convergence}. 
\end{proof}

\begin{proof}[Proof of Lemma \ref{lemmaflattening} part \ref{itemequivariant-d}]
By Lemma \ref{lemmaaligning}, $\rho(Qfs,3) = Q\rho(f,3)s$, so $\smagellan(Qfs,3) = Q\smagellan(f,3)$ and $\samundsen(Qfs,3) = Q\samundsen(f,3)$ and $p_i(Qfs,3) = p_{si}(f,3)$, so by Lemma \ref{lemma-siipointmap}, $\famundsen(Qfs,3) = Q\famundsen(f,3)\circ [z \mapsto sz^s]$.
Hence, $\sorpheus(Qfs,3) = s\sorpheus(f,3)^s$. 
In the case where $s = 1$ we have $r(t)(\sorpheus(Qfs,t)\cap\mathbf{D}) = r(t)(\sorpheus(f,t)\cap\mathbf{D})$ and analogously the other parts of $\sorpheus$ are unchanged, so $\sorpheus(Qfs,t) = \sorpheus(f,t) = s\sorpheus(f,t)^s$. 
In the case where $s = -1$ we have $r(t)(\sorpheus(Qfs,t)\cap\mathbf{D}) = r(t)(s\sorpheus(f,t)^s\cap\mathbf{D}) = s(\frac1{r(t)}(\overline{\sorpheus(f,t)\setminus\mathbf{D}}))^s$, 
and the other parts of $\sorpheus$ are permuted and transformed analogously, so $\sorpheus(Qfs,t) = s\sorpheus(f,t)^s$. 
Hence, $\smagellan(Qfs,t) = Q\famundsen(f,3)\circ [z \mapsto sz^s](s\sorpheus(f,t)^s) = Q\smagellan(f,t)$,
and $p_2(Qfs,t) = Qp_{s2}(f,t)$ and analogously for $p_{-2},p_1,p_{-1}$, 
so by Lemma \ref{lemma-sivpointmap}, 
$\fmagellan(Qfs,t) = Q\fmagellan(f,t)s$.
Thus, $\rho(Qfs,t) = Q\rho(Q,t)s$. 
\end{proof}

\begin{proof}[Proof of Lemma \ref{lemmaflattening} part \ref{itemstrong-d}]
Suppose $f \in \orth_3$.  Then, by Lemma \ref{lemmaaligning}, $\rho(f,3) = f$, so $\smagellan(f,3) = f(e_3^\bot)$ is a great circle and the 4 points $p_n(f,3) = f(e_n)$ for $n \in \{1,2,-1,-2\}$ are spaced uniformly at right angles around $\smagellan(f,3)$.
By Lemma \ref{lemma-sivpointmap}, $\fmagellan(f,3)$ is $f$ composed with stereographic projection through $e_{-3}$, so $\samundsen(f,3)$ is a great circle that is perpendicular to $\smagellan(f,3)$, 
so by Lemma \ref{lemma-siipointmap}, $\famundsen(f,3)$ is isometric to stereographic projection. 
Since the map $\famundsen(f,3)^{-1}$ sends $\samundsen(f,3)$ to $\sphere^1 \subset \mb{C}$ and sends the point $p_1$ to $1$, 
the map $\famundsen(f,3)^{-1}$ sends $\smagellan(f,3)$ to the line in $\overline{\mb{C}}$ that passes though $\sphere^1$ at $1$ at a right angle.
That is, $\famundsen(f,3)^{-1}$ sends $\smagellan(f,3)$ to $\overline{\mb{R}}$. 
Hence, $\sorpheus(f,t) = \overline{\mb{R}}$ for $t \in [3,4]_\mb{R}$, so $\smagellan(f,t) = f(e_3^\bot)$.
Also, $p_2(f,3)$ is orthogonal to $p_1(f,3)$, so $[\famundsen(f,3)]^{-1}(p_2(f,3)) = 0$, so $p_2(f,t) = \famundsen(f,3;0) = p_2(f,3)$, and similarly $p_{-2}(f,t) = p_{-2}(f,3)$, so $\fmagellan(f,t) = \fmagellan(f,3)$.
Therefore, $\rho(f,t) = \rho(f,3) = f$.
\end{proof}

\subsection{Divvying}

In the fifth stage, we deform the map to an isometry on the equator by linear interpolation in polar coordinates; see Figure \ref{figureFiveSix}. 

For $t \in (4,5]_\mb{R}$, we define the following.
Let $Q = Q_\mr{fin}(f) \in \orth_3$ such that $Qe_1 = p_1(f,4)$, $Qe_2 = p_2(f,4)$, and $Q$ has the same orientation as $f$. 
Eventually we will deformation retract to $Q$.  

Let 
\[ \rho(f,t;x) = [\fmagellan(f,4)] \left( z_1^{4-t} z_0^{5-t} \right) \]
where 
\begin{align*}
z_0 = \zeta_0(f;x) &= [[\fmagellan(f,4)]^{-1}\circ \rho(f,4)](x), \\
z_1 = \zeta_1(f;x) &= 
\begin{cases} 
z_0 \text{ if } z_0 \in \{0,\infty\} \\
|z_0| [[\fmagellan(f,4)]^{-1}\circ Q_\mr{fin}(f)\circ [\rho(f,4)]^{-1}\circ \fmagellan(f,4)] \left(\nicefrac{z_0}{|z_0|}\right)
 \text{ else}, 
\end{cases} \\ 
z_1^{4-t} z_0^{5-t} &= \mr{e}^{(4-t)\log(z_1)+(5-t)\log(z_0)}
\end{align*}
with the branch of the log with range $(-\infty,\infty)_\mb{R} + \im[0,2\pi)_\mb{R}$. 
Note that the cases in the definition of $z_1$ coincide if we simply regard $0$ or $\infty$ times an undefined unit complex number as $0$ or $\infty$ respectively.

\begin{lemma}\label{lemmaequatorial}
For $t \in (4,5]_\mb{R}$, $\rho(f,t) \in \hom(\sphere^2)$. 
Also, for $x \in e_3^\bot$, $\rho(f,5;x) = [Q_\mr{fin}(f)](x)$. 
\end{lemma}

\vbox{
\begin{lemma}\label{lemmadivvying} 
On $t \in [4,5]_\mb{R}$, the following hold.
\begin{enumerate}
\item \label{itemcontinuous-e}
In the sup-metric, $\rho(f,t)$ depends continuously on $f$ and $t$.
\item \label{itemequivariant-e}
For all $Q \in \orth_3$ and $s \in \{1,-1\}$, 
we have $\rho(Qfs,t) = Q\rho(f,t)s$. 
\item \label{itemstrong-e}
If $f \in \orth_3$, then $\rho(f,t) = f$.
\end{enumerate}
\end{lemma}
}

\begin{proof}[Proof of Lemma \ref{lemmaequatorial}]
For $x \in e_3^\bot$, we have $\rho(f,4;x) \in \smagellan(f,4)$, so $z_0 \in \sphere^1$, so $|z_0| = 1$.
Hence, $\zeta_1(f;x)$ restricted to $x \in e_3^\bot$ is a composition of homeomorphisms;
namely 
\begin{align*}
\zeta_1(f;x) 
&= [\fmagellan(f,4)]^{-1}\circ Q_\mr{fin}(f)\circ [\rho(f,4)]^{-1}\circ \fmagellan(f,4)] \left( [[\fmagellan(f,4)]^{-1}\circ \rho(f,4)](x)  \right) \\
&= [\fmagellan(f,4)]^{-1}\circ Q_\mr{fin}(f)](x).
\end{align*}

Likewise, the map $\zeta_1(f)\circ [\zeta_0(f)]^{-1}$, which sends $z_0$ to $z_1$, is a self-homeomorphism of $\sphere^1$, and since $Q_\mr{fin}(f)$ and $\rho(f,4)$ have the same orientation, this map is orientation preserving. 
Also, 
\begin{align*} 
[\zeta_1(f)\circ [\zeta_0(f)]^{-1}](1) 
&= [[\fmagellan(f,4)]^{-1}\circ Q_\mr{fin}(f)\circ [\rho(f,4)]^{-1}\circ \fmagellan(f,4)](1) \\
&= [[\fmagellan(f,4)]^{-1}\circ Q_\mr{fin}(f)\circ [\rho(f,4)]^{-1}]](p_1(f,4)) \\
&= \fmagellan(f,4;p_1(f,4)) \\
&= 1,
\end{align*}
so the phase of $z_1$ as a function of the phase of $z_0$, which is a the map given by 
$\zeta_\mr{ph}(\theta) = [\theta \mapsto -\im[\log \circ \zeta_1(f)\circ [\zeta_0(f)]^{-1}](\mr{e}^{\im\theta})]$, 
is strictly increasing, and the restriction to $(0,2\pi)_\mb{R}$ is a self-homeomorphism of $(0,2\pi)_\mb{R}$,
provided that we choose the appropriate branch of the log.   
Consider the map 
\[\widetilde\rho(f,t) = [\fmagellan(f,4)]^{-1}\circ \rho(f,t)\circ [\rho(f,4)]^{-1} \circ  \fmagellan(f,4),\] 
which sends $z_0$ to $z_1^{4-t} z_0^{5-t}$.  
%
The phase of $\widetilde\rho(f,t;z_0)$ as a function of the phase of $z_0$, 
which is given by 
$[\theta \mapsto -\im\log(\widetilde\rho(f,t;e^{\im\theta})) = (4-1)\zeta_\mr{ph}(\theta) +(5-t)\theta$, 
is also a strictly increasing self-homeomorphism $(0,2\pi)_\mb{R}$. 
Hence, the restriction of $\widetilde\rho(f,t)$ to $\sphere^1$ is a self-homeomorphism of $\sphere^1$. 
The change in $\widetilde\rho(f,t)$ with respect to change in magnitude is defined to be the identity, so $\widetilde\rho(f,t)$ is a self-homeomorphism of $\overline{\mb{C}}$ with the metric induced by stereographic projection. 
Thus, $\rho(f,t) \in \hom(\sphere^2)$, as it should be. 

For the last part of the lemma, we have 
\begin{align*}
\rho(f,5;x)
&= [\fmagellan(f,4)] \left( z_1 \right) \\
&= [\fmagellan(f,4)] \left( [\fmagellan(f,4)]^{-1}\circ Q_\mr{fin}(f)](x) \right) \\
&= [Q_\mr{fin}(f)](x). \qedhere
\end{align*}
\end{proof}

\begin{proof}[Proof of Lemma \ref{lemmadivvying} part \ref{itemcontinuous-e}]
If we extend the definition above for $\rho(f,t)$ on $t \in (4,5]_\mb{R}$ to $t=4$, we get 
\begin{align*}
[\fmagellan(f,4)] \left( z_1^{4-t} z_0^{5-t} \right)
&= [\fmagellan(f,4)] \left(z_0\right) \\
&= [[\fmagellan(f,4)] \circ [\fmagellan(f,4)]^{-1}\circ \rho(f,4)](x) \\
&= \rho(f,4;x).
\end{align*} 
Thus, it suffices to show that $\rho(f,t)$ continuous on $t \in [4,5]_\mb{R}$ using the definition above.
Let 
\begin{align*}
\phi_\mr{Q}(f,t) 
&= Q_\mr{fin}(f)^{-1}\circ\rho(f,t)\circ [\rho(f,4)]^{-1} \circ Q_\mr{fin}(f) \\
&= [Q_\mr{fin}(f)^{-1} \circ  \fmagellan(f,4)] \circ \widetilde \rho(f,t) \circ [Q_\mr{fin}(f)^{-1} \circ  \fmagellan(f,4)]^{-1}. \end{align*}
On the restriction to $e_3^\bot$, $\rho_\mr{Q}(f,t)$ is a composition of functions and inverses of functions that depend continuously in the sup-metric on $f$ and $t$,  
so by Lemma \ref{lemma-invariable-convergence}, the map $\rho_\mr{Q}(f,t)$ is also continuous in the sup-metric on the restriction to $e_3^\bot$.  
Since change in $\widetilde \rho(f,t)$ with respect to change in magnitude is the identity, 
and $Q_\mr{fin}(f)^{-1} \circ  \fmagellan(f,4)$ is the stereographic projection though $e_3$, 
we have that the change in $\phi_\mr{Q}$ 
with respect to change along a great circle though $e_3$ is the identity.
Hence, $\rho_\mr{Q}(f,t)$ is continuous in the sup-metric on all of $\sphere^2$.
Thus, $\rho(f,t)$ is continuous in the sup-metric
\end{proof}

\begin{proof}[Proof of Lemma \ref{lemmadivvying} part \ref{itemequivariant-e}]
For arbitrary $Q \in \orth_3$ and $t \in [4,5]_\mb{R}$ we have the following. 
\begin{align*}
\zeta_0(Qfs) 
&= [\fmagellan(Qfs,4)]^{-1}\circ \rho(Qfs,4) \\
&= s[\fmagellan(f,4)]^{-1}Q^{-1}\circ Q\rho(f,4)s \\
&= s\zeta_0(f)s.
\end{align*}
Since $p_i(Qfs,4) = Qp_{si}(f,4)$, 
and $QQ_\mr{fin}(f)s$ sends $e_i$ to $Qp_{si}(f,4)$ for $i \in \{1,2,-1,-2\}$ 
and has the same orientation as $Qfs$, 
we have 
$Q_\mr{fin}(Qfs) = QQ_\mr{fin}(f)s$. 
\begin{align*}
\zeta_1(Qfs;x) 
&= |\zeta_0(Qfs;x)| [[\fmagellan(Qfs,4)]^{-1}\circ Q_\mr{fin}(Qfs)\circ [\rho(Qfs,4)]^{-1}\circ \fmagellan(Qfs,4)] \left(\tfrac{\zeta_0(Qfs;x)}{|\zeta_0(Qfs;x)|}\right) \\ 
&= |\zeta_0(f;sx)| [s[\fmagellan(f,4)]^{-1}Q^{-1}\circ QQ_\mr{fin}(f)s\circ s[\rho(f,4)]^{-1}Q^{-1}\circ Q\fmagellan(f,4)s] \left(\tfrac{s\zeta_0(f;sx)}{|\zeta_0(f;sx)|}\right) \\ 
&= s\zeta_1(f;sx)
\end{align*}
\begin{align*}
\rho(Qfs,t;x) 
&= [\fmagellan(Qfs,4)] \left( \mr{e}^{(4-1)\log(\zeta_1(Qfs;x)) +(5-t)\log(\zeta_0(Qfs;x))} \right) \\
&= [Q\fmagellan(f,4)s] \left( \mr{e}^{(4-1)\log(s\zeta_1(f;sx)) +(5-t)\log(s\zeta_0(f;sx))} \right) \\
&= [Q\fmagellan(f,4)s]\left( \mr{e}^{(4-1)\log(\zeta_1(f;sx)) +(5-t)\log(\zeta_0(f;sx)) +\log(s)} \right) \\
&= Q\rho(f,t;sx). \qedhere
\end{align*}
\end{proof}

\begin{proof}[Proof of Lemma \ref{lemmadivvying} part \ref{itemstrong-e}]
Suppose $f \in \orth_3$.  Then, by Lemma \ref{lemmaflattening}, $\rho(f,4) = f$, and since $\rho(f,4)$ satisfies the defining properties of $Q_\mr{fin}$, we have $Q_\mr{fin}(f) = f$ as well, so $z_1 = z_0$, so $z_1^{4-t}z_0^{4-t} = z_0$, so $\rho(f,t) = \rho(f,4) = f$. 
\end{proof}

\subsection{Combing}

In the sixth stage, we use Alexander's trick on each hemisphere to deform the map to an isometry; see Figure \ref{figureFiveSix}. 

For $t \in (5,6]_\mb{R}$, we define the following. 
Let 
\[ J(f,r;z) = \begin{cases}
rg\left(f;\tfrac{z}{r}\right) & 0 \leq |z| < r \\ 
z & r \leq |z| \leq \tfrac{1}{r} \\
\tfrac{1}{r}g(f;rz) & \tfrac{1}{r} < |z| \leq \infty 
\end{cases}  \]
where $g(f) = [\fmagellan(f,4)]^{-1} \circ \rho(f,5) \circ [Q_\mr{fin}(f)]^{-1} \circ \fmagellan(f,4)$, 
and let 
\[ \rho(f,t) = \fmagellan(f,4) \circ J(f,6-t) \circ [\fmagellan(f,4)]^{-1} \circ Q_\mr{fin}(f). \]

\begin{lemma}\label{lemmahomeo56}
For $t \in (5,6]_\mb{R}$, $\rho(f,t) \in \hom(\sphere^2)$. 
\end{lemma}

\begin{lemma}\label{lemmarhoqfin}
$\rho(f,6) = Q_\mr{fin}(f) \in \orth_3$. 
\end{lemma}

\vbox{
\begin{lemma}\label{lemmacombing} 
On $t \in [4,5]_\mb{R}$, the following hold.
\begin{enumerate}
\item \label{itemcontinuous-f}
In the sup-metric, $\rho(f,t)$ depends continuously on $f$ and $t$.
\item \label{itemequivariant-f}
For all $Q \in \orth_3$ and $s \in \{1,-1\}$, 
we have $\rho(Qfs,t) = Q\rho(f,t)s$. 
\item \label{itemstrong-f}
If $f \in \orth_3$, then $\rho(f,t) = f$.
\end{enumerate}
\end{lemma}
}

\begin{proof}[Proof of Lemma \ref{lemmahomeo56}]
We claim that the restriction of $g(f)$ to $\sphere^1$ is the identity function.
For $z \in \sphere^1$, we have $\famundsen(f,4;z) \in \samundsen(f,4)$, so $[[Q_\mr{fin}(f)]^{-1} \circ \famundsen(f,4)](z) \in e_3^{-1}$, so 
by Lemma \ref{lemmaequatorial}, $[\rho(f,5) \circ [Q_\mr{fin}(f)]^{-1} \circ \fmagellan(f,4)](z) = \fmagellan(f,4;z)$, so $g(f;z) = [[\fmagellan(f,4)]^{-1} \circ \fmagellan(f,4)](z) = z$, 
so the claim holds. 

Hence, if $|z| = r$ then $rg(f;\frac{z}{r}) = z$, and if $|z| = \frac1r$ then $\frac1r g(f;rz) = z$. 
Therefore, $J(f,r)$ is continuous along the boundary between cases, and similarly for $[J(f,r)]^{-1}$, so $J(f,r)$ is a homeomorphism, so $\rho(f,t)$ is a homeomorphism.
\end{proof}

\begin{proof}[Proof of Lemma \ref{lemmarhoqfin}]
At $t=6$, we have $J(f,0) = \id_\mb{C}$, so 
$\rho(f,6) = \fmagellan(f,4) \circ [\fmagellan(f,4)]^{-1} \circ Q_\mr{fin}(f) = Q_\mr{fin}(f)$.
\end{proof}

\begin{proof}[Proof of Lemma \ref{lemmacombing} part \ref{itemcontinuous-f}]
By Lemmas \ref{lemmasupconvergence}, \ref{lemmaflattening}, and \ref{lemmadivvying}, 
$g(f)$ depends continuously on $f$, so $J(f,r)$ also depends continuously on $f$ and $r$, 
so $\rho(f,t)$ depends continuously on $f$ and $t$.
\end{proof}

\begin{proof}[Proof of Lemma \ref{lemmacombing} part \ref{itemequivariant-f}]
For $Q \in \orth_3$ and $t \in (5,6]_\mb{R}$, we have the following.
\begin{align*} 
g(Qfs)
&= [\fmagellan(Qfs,4)]^{-1} \circ \rho(Qfs,5) \circ [Q_\mr{fin}(Qfs)]^{-1} \circ \fmagellan(Qfs,4) \\
&= s[\fmagellan(f,4)]^{-1}Q^{-1} \circ Q\rho(f,5)s \circ s[Q_\mr{fin}(f)]^{-1}Q^{-1} \circ Q\fmagellan(f,4)s \\
&= sg(f)s,
\end{align*}
so $J(Qfs,r) = sJ(f,r)s$, so 
\begin{align*} 
\rho(Qfs,t)
&= \fmagellan(Qfs,4) \circ J(Qfs,6-t) \circ [\fmagellan(Qfs,4)]^{-1} \circ Q_\mr{fin}(Qfs) \\
&= Q\fmagellan(f,4)s \circ sJ(f,6-t)s \circ s[\fmagellan(f,4)]^{-1}Q^{-1} \circ QQ_\mr{fin}(f)s \\
&= Q\rho(f,t)s. \qedhere
\end{align*}
\end{proof}

\begin{proof}[Proof of Lemma \ref{lemmacombing} part \ref{itemstrong-f}]
Suppose $f \in \orth_3$.  Then, $\rho(f,5) = Q_\mr{fin}(f) = f$, so $g(f) = \id_\mb{C}$, so $J(f,r) = \id_\mb{C}$, so $\rho(f,t) = f$. 
\end{proof}

\subsection{Putting it together}

\begin{proof}[Proof of Theorem \ref{theoremDeformation}]
By Lemma \ref{lemmabalancing} Part \ref{itemstart}, we have $\rho(f,0) =f$, 
so by Part \ref{itemcontinuous-a} of Lemmas \ref{lemmabalancing}, \ref{lemmauntangling}, \ref{lemmaaligning}, \ref{lemmaflattening}, \ref{lemmadivvying}, \ref{lemmacombing}, $\rho$ is a continuous deformation of $\hom(\sphere^2)$.
By Part \ref{itemstrong-a} of the above lemmas and Lemma \ref{lemmarhoqfin}, $\rho$ is a strong deformation retraction to $\orth_3$, and by Part \ref{itemequivariant-a} of the above lemmas, $\rho$ is $(\orth_3\times \mb{Z}_2)$-equivariant.  
\end{proof}

\section{Nullity preserving homeomorphisms}\label{sectionDeformationN}

The goal of this section is to prove Theorem \ref{theoremDeformationN}. 
The construction will be similar, except that we will use a different definition for $\fstitch$. 
We will first define a function $\fradmul$ that extends a homeomorphism $f$ of the circle to a homeomorphism the disk that is smooth almost everywhere.  We will extend $f$ to the interior of the disk radially, but on each concentric circle, we mollify $f$ by convolution with a smooth bump function that approaches the Dirac delta distribution at the boundary of the disk.

Let $\bump(\eps)$ be a normalized smooth bump function with support on $[-\eps,\eps]_\mb{R}$. 
Specifically, let 
\[\bump(\eps;x) = 
\begin{cases}
\frac{1}{c\eps}\e^{\frac{-1}{1-(\frac{x}{\eps})^2}} & |x| < \eps \\
0 & |x| \geq \eps
\end{cases}
\]
where $c = \int_{-1}^{1}\e^{\frac{-1}{1-x^2}} \mathrm{d}x $.

Given a homeomorphism $f : \sphere^1 \to \sphere^1$ such that $f(1) = 1$, 
let $\fradmul : \mathbf{D} \to \mathbf{D}$ be defined as follows.
Let $g(\theta) = -\im \log \left(f\left(\e^{\im \theta}\right)\right)$,
and for $r\in(0,1)_\mb{R}$ let 
\begin{align*}
\fradmul(f;r\e^{\im\theta})
&= r\e^{\im \left[\bump(1-r) * g \right](\theta)}  \\
&= r\exp{\left(\displaystyle \frac{1}{c(1-r)} \int_{r-1}^{1-r} \e^{\frac{-1}{1-(\frac{x}{1-r})^2}} \log\left(f\left(\e^{\im (\theta-x)}\right)\right) \mathrm{d}x \right)},
\end{align*}
and $\fradmul(f;\e^{\im\theta}) = f(\e^{\im\theta})$ and $\fradmul(f;0) = 0$.

\begin{lemma}\label{lemmaFradmul}
If $f$ is a homeomorphism of the circle, then $\fradmul(f)$ satisfies the following.
\begin{enumerate}
\item
$\fradmul(f)$ preserves nullity. 
\item
$\fradmul(f)$ in the sup-metric depends continuously on $f$ in the sup-metric.
\item 
$\fradmul(f\circ [z \mapsto \nicefrac1z]) = \fradmul(f \circ [z \mapsto \overline{z}]) = \fradmul(f) \circ [z \mapsto \overline{z}]$
\end{enumerate}
\end{lemma}

\begin{proof}
Observe that $[\bump(\eps)*g](\theta) \in [\min,\max]_\mb{R}(g([\theta-\eps,\theta+\eps]_\mb{R}))$ by comparing integrals, since for any constant function $c$ we have $[\bump(\eps)*c] = c$. 
Hence, $\fradmul(f)$ is continuous at the boundary of the disk. 

Suppose that $f$ is orientation preserving.
Then, $g$ is strictly increasing and quasiperiodic with $g(\theta+2\pi) = g(\theta)+2\pi$.
Hence, $\bump(\eps)*g$ is also strictly increasing and quasiperiodic with $[\bump(\eps)*g](\theta+2\pi) = [\bump(\eps)*g](\theta)+2\pi$. Hence, $e^{\im [\bump(\eps)*g](\theta)}$ is an orientation preserving homeomorphism of $\sphere^1$, so $\fradmul(f)$ is injective. 
Furthermore, $\partial_\theta \bump(\eps;\theta)$ is positive on $(-\eps,0)_\mb{R}$ and negative on $(0,\eps)_\mb{R}$, and since $g(\theta + \eps) > g(\theta)$, this implies 
$\partial_\theta [\bump(\eps)*g](\theta) = [\partial_\theta \bump(\eps;\theta)] *  g$ is always positive. 
Similarly, $\fradmul(f)$ is injective and $\partial_\theta [\bump(\eps)*g](\theta)$ is always negative in the case where $f$ is orientation reversing. 

In either case, $\partial_\theta \fradmul(f;r\e^{\im\theta})$ never vanishes for $r \in (0,1)_\mb{R}$. 
Also, 
$\partial_r |\fradmul(f;r\e^{\im\theta})| = \partial_r r = 1$, 
so $\partial_r \fradmul(f;r\e^{\im\theta})$ never vanishes.
By the dominated convergence theorem, the partial derivatives of $\fradmul(f)$ are continuous, so $\fradmul(f)$ is continuously differentiable on $\mathbf{D}^\circ\setminus0$.
Since $\fradmul(f)$ has no critical points in $\mathbf{D}^\circ\setminus0$, 
$\fradmul(f)^{-1}$ is continuously differentiable on $\mathbf{D}^\circ\setminus0$ by the inverse function theorem. 
Hence, both $\fradmul(f)$ and $\fradmul(f)^{-1}$ send null sets to null sets, so the first part holds.

For the second part,
consider $\eps > 0$ and $f_1, f_\infty$ such that $|f_1 - f_\infty| < \eps$,
and let $g_k$ be defined analogously. 
Then, $|g_1 - g_\infty| < 2\eps$, 
so for all $r$, we have $|[\bump(1-r)*g_1]-[\bump(1-r)*g_\infty]| = |\bump(1-r)*(g_1-g_\infty)| < 2\eps$, 
so $|\fradmul(f_1)-\fradmul(f_\infty)| < 2\eps$.
Hence, $\fradmul(f)$ in the sup-metric depends continuously on $f$ in the sup-metric.

For the third part,
observe that the maps $[z \mapsto \nicefrac1z]$ and $[z \mapsto \overline{z}]$ both act on $\sphere^1$ by reflecting the circle across the real line, so these are actually the same map, so the first equality holds. 
Also, 
$
-\im \log \left([f \circ [z \mapsto \nicefrac1z]] \left(e^{\im \theta}\right)\right) 
= -\im \log \left(f \left(e^{-\im \theta}\right)\right) 
= g(-\theta)
$,
and since $\bump(1-r)$ is symmetric about $0$, we have 
\begin{align*}
\fradmul(f \circ [z \mapsto \nicefrac1z];re^{\im\theta}) 
&= re^{\im \left[\bump(1-r) * (g \circ [\psi \mapsto -\psi] ) \right](\theta)} \\
&= re^{\im \left[\bump(1-r) * g \right](-\theta)} \\ 
&= \fradmul(f;re^{\im(-\theta)}) \\
&= \fradmul(f;\overline{re^{\im\theta}}), \\
\end{align*}
so 
$\fradmul(f\circ [z \mapsto \nicefrac1z]) = \fradmul(f) \circ [z \mapsto \overline{z}]$. 
\end{proof}

We redefine $\fstitch(f_1,f_2)$ as follows. 
Given embeddings $f_i : \mathbf{D} \to \sphere^2$ with $p_1 = f_1(1)=f_2(1)$,  $p_{-1} = f_1(-1)=f_2(-1)$,  and $S = f_1(\sphere^1) = f_2(\sphere^1)$, 
let $f_0$ be the same as before, which is
\[ 
f_0^{-1}: S \to \mathbf{S}^1, 
\quad f_0^{-1}(x) 
= \sqrt{\frac{f_1^{-1}(x)}{f_2^{-1}(x)}},
\] 
and let 
\begin{equation}\label{equationStitchNull}
\fstitch(f_1,f_2;z) = \begin{cases}
[f_1 \circ \fradmul(f_1^{-1} \circ f_0)](z) & |z| < 1 \\
f_0(z) & |z| = 1 \\
[f_2 \circ \fradmul(f_2^{-1} \circ f_0)](\nicefrac{1}{\overline{z}}) & |z| > 1. 
\end{cases}
\end{equation}

\begin{lemma}\label{lemmaStitchN}
Lemma \ref{lemma-stitch} also holds for $\fstitch$ redefined by (\ref{equationStitchNull}). 
Additionally:
\begin{enumerate}
\setcounter{enumi}{4}
\item\label{item-stitchNull}
If both $f_1$ and $f_2$ preserve nullity, then $\fstitch(f_1,f_2)$ preserves nullity.
\end{enumerate}
\end{lemma}

\begin{proof}
Observe that in the case where $|z| = 1$, we have 
\begin{align*}
f_1 \circ \fradmul(f_1^{-1} \circ f_0)](z) = [f_1 \circ f_1^{-1} \circ f_0](z) &= f_0(z), \\
[f_2 \circ \fradmul(f_2^{-1} \circ f_0)](\nicefrac{1}{\overline{z}}) = [f_2 \circ f_2^{-1} \circ f_0](z) &= f_0(z),
\end{align*}
so the defining formulas for $\fstitch(f_1,f_2)$ in all 3 case coincide when $|z| = 1$. 
Hence, $\fstitch(f_1,f_2)$ is indeed a homeomorphism, 
and part \ref{item-stitchNull} holds by Lemma \ref{lemmaFradmul}.

Observe that $f_0$, i.e., the restriction of $\fstitch(f_1,f_2)$ to the unit circle, is defined in the same way as before, so Lemma \ref{lemma-stitch} still holds for $f_0$.  
Consider embeddings $f_{i,k}$ in the domain of $\fstitch$ such that $f_{i,k} \to f_{i,\infty}$ uniformly.
Just as before $f_{i,k}^{-1} \circ f_{0,k} \to f_{i,\infty}^{-1} \circ f_{0,\infty}$ invariably, 
so $\fradmul(f_{i,k}^{-1} \circ f_{0,k}) \to \fradmul(f_{i,\infty}^{-1} \circ f_{0,\infty})$ by Lemma \ref{lemmaFradmul},
so $\fstitch(f_{1,k},f_{2,k}) \to \fstitch(f_{1,\infty},f_{2,\infty})$,
so part \ref{item-stitchcontinuity} holds, i.e., $\fstitch$ is continuous. 
For $Q \in \sorth_3$, we have $(Qf_i) \circ \fradmul((Qf_i)^{-1} \circ (Qf_0)) = Q(f_i \circ \fradmul(f_i^{-1} \circ f_0))$, 
so just as before, $\fstitch(Q(f_1,f_2)) = Q\fstitch(f_1,f_2)$, 
which means part \ref{item-stitchsorth} holds. 
Similarly, $\fstitch(-(f_1,f_2);z) = -\fstitch(f_1,f_2;\overline{z})$, 
which means part \ref{item-stitchnegate} holds, 
and $\fstitch(f_2,f_1;z) = \fstitch(f_1,f_2;\nicefrac1z)$ on $z \in \sphere^1$, which means 
part \ref{item-stitchswap} holds on $z \in \sphere^1$.
In the case where $ |z| < 1$, we have $|\nicefrac{1}{z}| > 1$, so by Lemma \ref{lemmaFradmul}, we have  
\begin{align*} 
\fstitch(f_1,f_2;\nicefrac{1}{z}) 
&= [f_2 \circ \fradmul(f_2^{-1} \circ \fstitch(f_1,f_2))](\overline{z}) \\
&= [f_2 \circ \fradmul(f_2^{-1} \circ \fstitch(f_2,f_1) \circ [w \mapsto \nicefrac1w])](\overline{z}) \\
&= [f_2 \circ \fradmul(f_2^{-1} \circ \fstitch(f_2,f_1))](z) \\
&= \fstitch(f_2,f_1;z),
\end{align*}
and likewise in the case where $|z|> 1$, 
so part \ref{item-stitchswap} holds.
\end{proof}

\begin{proof}[Proof of Theorem \ref{theoremDeformationN}]
Let $\rho$ be similar to the deformation retraction from Section \ref{sectionDeformation}, except with $\fstitch$ redefined by equation (\ref{equationStitchNull}).
By Lemma \ref{lemmaStitchN}, the redefined map $\fstitch$ still satisfies all relevant properties needed for the proof of Theorem \ref{theoremDeformation}, so the redefined $\rho$ is still a strong $(\orth_3\times\mb{Z}_2)$-equivariant deformation retraction from $\hom(\sphere^2)$ to $\orth_3$. 
It only remains to show that if $f \in \hom_\mathrm{N}(\sphere^2)$, then $\rho(f,t) \in \hom_\mathrm{N}(\sphere^2)$.

For $f \in \hom_\mathrm{N}(\sphere^2)$, the curve $\smagellan$ is a null set, 
so $\fdivp(\smagellan,p_1,p_2,p_{-1},p_{-2})$ and $\fdivp(\smagellan,p_1,p_{-2},p_{-1},p_{2})$ 
preserve nullity, so by Lemma \ref{lemmaStitchN}, 
$\fmagellan \in \hom_\mathrm{N}(\sphere^2)$.  Also, $\famundsen \in \hom_\mathrm{N}(\sphere^2)$, so $\rho$ at each step is a composition of maps in $\hom_\mathrm{N}(\sphere^2)$, so $\rho(f,t) \in \hom_\mathrm{N}(\sphere^2)$. 
\end{proof}

\bibliographystyle{plain}
\bibliography{aho2-4}

\begin{thebibliography}{10}

\bibitem{angenent1991parabolic}
Sigurd Angenent.
\newblock Parabolic equations for curves on surfaces: Part ii. intersections,
  blow-up and generalized solutions.
\newblock {\em Annals of Mathematics}, pages 171--215, 1991.

\bibitem{babson1993combinatorial}
Eric~Kendall Babson.
\newblock {\em A combinatorial flag space}.
\newblock PhD thesis, Massachusetts Institute of Technology, 1993.

\bibitem{biss2003homotopy}
Daniel~K Biss.
\newblock The homotopy type of the matroid grassmannian.
\newblock {\em Annals of mathematics}, 158(3):929--952, 2003.

\bibitem{biss2009erratum}
Daniel~K Biss.
\newblock Erratum to “{T}he homotopy type of the matroid {G}rassmannian”.
\newblock {\em Annals of mathematics}, 170(1):493--493, 2009.

\bibitem{dobbins2021continuous}
Michael~Gene Dobbins.
\newblock Continuous dependence of curvature flow on initial conditions.
\newblock {\em arXiv:2106.08907}, 2021.

\bibitem{dobbins2021grassmannians}
Michael~Gene Dobbins.
\newblock Grassmannians and pseudosphere arrangements.
\newblock {\em Journal de l’{\'E}cole polytechnique—Math{\'e}matiques},
  8:1225--1274, 2021.

\bibitem{folkman1978oriented}
Jon Folkman and Jim Lawrence.
\newblock Oriented matroids.
\newblock {\em Journal of Combinatorial Theory, Series B}, 25(2):199--236,
  1978.

\bibitem{friberg1973topological}
Bjorn Friberg.
\newblock A topological proof of a theorem of {K}neser.
\newblock {\em Proceedings of the American Mathematical Society},
  39(2):421--426, 1973.

\bibitem{gage1990curve}
Michael~E Gage.
\newblock Curve shortening on surfaces.
\newblock {\em Annales scientifiques de l'Ecole normale sup{\'e}rieure},
  23(2):229--256, 1990.

\bibitem{hamstrom1965homotopy}
Mary-Elizabeth Hamstrom.
\newblock Homotopy properties of the space of homeomorphisms on
  {$\mathrm{P}^2$} and the {K}lein bottle.
\newblock {\em Transactions of the American Mathematical Society},
  120(1):37--45, 1965.

\bibitem{kneser1926deformationssatze}
Hellmuth Kneser.
\newblock Die {D}eformationssatze der einfach zusammenhangenden {F}läschen.
\newblock {\em Mathematische Zeitschrift}, 25:362--372, 1926.

\bibitem{lauer2016evolution}
Joseph Lauer.
\newblock The evolution of jordan curves on {S}${}^2$ by curve shortening flow.
\newblock {\em arXiv preprint arXiv:1601.05704}, 2016.

\bibitem{liu2020counterexample}
Gaku Liu.
\newblock A counterexample to the extension space conjecture for realizable
  oriented matroids.
\newblock {\em Journal of the London Mathematical Society}, 101(1):175--193,
  2020.

\bibitem{mnev1993combinatorial}
Nicolai~E. Mn\"{e}v and G\"{u}nter~M. Ziegler.
\newblock Combinatorial models for the finite-dimensional {G}rassmannians.
\newblock {\em Discrete \& Computational Geometry}, 10(3):241--250, 1993.

\bibitem{osgood1903jordan}
William~F Osgood.
\newblock A {J}ordan curve of positive area.
\newblock {\em Transactions of the American Mathematical Society},
  4(1):107--112, 1903.

\bibitem{pommerenke1992boundary}
Christian Pommerenke.
\newblock {\em Boundary Behaviour of Conformal Maps}.
\newblock Springer, 1992.

\end{thebibliography}

\end{document}